\documentclass{amsart}
\usepackage{amssymb, amsfonts, amsthm}
\usepackage{graphics, epstopdf, pstricks, psfrag}

\title{Weak separation and plabic graphs}

\author{Suho Oh, Alexander Postnikov, David E.\ Speyer}

\address{Department of Mathematics, University of Michigan, 
         Ann Arbor, MI 48109}
\address{Department of Mathematics, MIT, Cambridge, MA 02139}
\address{Department of Mathematics, University of Michigan, 
         Ann Arbor, MI 48109}


\thanks{The second author was supported in part by NSF CAREER Award DMS-0504629.}
\thanks{The third author was supported in part by a Clay Research Fellowship.}

\theoremstyle{plain}
\newtheorem{theorem}{Theorem}[section]
\newtheorem{conjecture}[theorem]{Conjecture}
\newtheorem{proposition}[theorem]{Proposition}
\newtheorem{prop}[theorem]{Proposition}
\newtheorem{lemma}[theorem]{Lemma}
\newtheorem{corollary}[theorem]{Corollary}

\theoremstyle{definition}
\newtheorem{definition}[theorem]{Definition}

\theoremstyle{remark}
\newtheorem{remark}[theorem]{Remark}
\newtheorem{example}[theorem]{Example}

\newcommand{\Ch}{\mathcal{H}}
\newcommand{\PCh}{\mathcal{PH}}

\newcommand{\RR}{\mathbb{R}}

\newcommand{\F}{\mathcal{F}}

\newcommand{\newword}[1]{\textbf{\emph{#1}}}

\newcommand{\C}{\mathcal{C}}
\newcommand{\I}{\mathcal{I}}
\newcommand{\B}{\mathcal{B}}
\newcommand{\W}{\mathcal{W}}

\def\f{\mathcal{F}}
\def\M{\mathcal{M}}
\def\G{\mathcal{G}}

\newcommand{\LZprec}{\prec}
\newcommand{\LZpar}{\parallel_{\mathrm{LZ}}}
\newcommand{\pad}[1]{\texttt{pad}(#1)}

\newcommand{\lw}{\sqsubset}
\newcommand{\lweq}{\sqsubseteq}
\begin{document}

\begin{abstract}
Leclerc and Zelevinsky described quasicommuting families of quantum minors in
terms of a certain combinatorial condition, called {\it weak separation.}  They
conjectured that all maximal by inclusion weakly separated collections of
minors have the same cardinality, and that they can be
related to each other by a sequence of mutations.  

On the other hand, Postnikov studied {\it total positivity\/} on the
Grassmannian.  He described a stratification of the totally nonnegative
Grassmannian into {\it positroid strata,} and constructed their
parametrization using {\it plabic graphs.}

In this paper we link the study of weak separation to plabeic graphs.  We extend the notion of weak separation to positroids.
We generalize the conjectures of Leclerc and Zelevinsky, and related ones of Scott,
and prove them.   We show that the maximal weakly
separated collections in a positroid are in bijective correspondence with
the plabic graphs.   This correspondence allows us to use 
the combinatorial techniques of positroids and plabic graphs to prove
the (generalized) purity and mutation connectedness conjectures.
\end{abstract}

\maketitle

\section{Introduction}

%All definitions here are repeated and expanded on in the following sections.

Leclerc and Zelevinsky \cite{LZ}, in their study of quasicommuting families of quantum minors,
introduced the following notion of weakly separated sets.

Let $I$ and $J$ be two subsets of $[n]:=\{ 1,2, \ldots, n \}$.  
Leclerc and Zelevinsky \cite{LZ}, defined $I$ and $J$ to be {\it weakly separated\/} if either
\begin{enumerate}
 \item $|I| \leq |J|$ and $I \setminus J$ can be partitioned 
as $I_1 \sqcup I_2$ such that $I_1 \prec J \setminus I \prec I_2$, or
\item $|J| \leq |I|$ and $J \setminus I$ can be partitioned 
as $J_1 \sqcup J_2$ such that $J_1 \prec I \setminus J \prec J_2$,
\end{enumerate}
where $A\prec B$ indicates that every element of $A$ is less than every element of $B$.

Leclerc and Zelevinsky proved that, for any collection $\C$ of pairwise weakly separated subsets of $[n]$, 
one has $|\C| \leq \binom{n}{2}+n+1$.  Moreover, they made the following {\it Purity Conjecture.}

%\begin{proposition}
%{\rm \cite{LZ}}
%If $\C$ is a collection of subsets of $[n]$, each of which are pairwise weakly separated from each other, then $|\C| \leq \binom{n}{2}+n+1$.
%\end{proposition}

%Leclerc and Zelevinsky made the following {\it Purity Conjecture.}

\begin{conjecture}
\label{conj:LZ}
{\rm \cite{LZ}}
 If $\C$ is a collection of subsets of $[n]$, each of which are pairwise weakly separated from each other, and such that $\C$ is not contained in any larger collection with this property, then $|\C| = \binom{n}{2}+n+1$.
\end{conjecture}

The above notion of weak separation is related to the study of the Pl\"ucker 
coordinates on the flag manifold, see \cite{LZ}.
Similarly, in the context of Pl\"ucker coordinates on the Grassmannian,
it is natural to study weak separation of $k$ element subsets of $[n]$, for fixed $n$ and $k$. 
Observe that, when $|I|=|J|$, the definition of weak separation becomes 
invariant under cyclic shifts of $[n]$.  Indeed, in this case $I$ and $J$ are weakly 
separated if and only if after an appropriate cyclic shift $(I\setminus J) \prec (J\setminus I)$. 

%Postnikov \cite{Post} and Scott \cite{Scott}
% made similar studies of the Grassmannian. 

Scott \cite{Scott} proved that, for any collection $\C\subset \binom{[n]}{k}$ of pairwise weakly separated 
$k$ element subsets on $[n]$, one has $|\C| \leq k(n-k)+1$.
Here $\binom{[n]}{k}$ denotes the set of $k$ element subsets of $[n]$. 
Moreover, Scott made the following conjecture.

%\begin{proposition}
%\cite{Scott}
%If $\C \subset \binom{[n]}{k}$ is a collection of $k$ element subsets of $[n]$, 
%each of which are pairwise weakly separated from each other, then $|\C| \leq k(n-k)+1$.
%\end{proposition}

\begin{conjecture}
\cite{Scott}
\label{conj:Scott}
If $\C \subset  \binom{[n]}{k}$ 
is a collection of $k$ element subsets of $[n]$, each of which are pairwise weakly separated from each other, and such that $\C$ is not contained in any larger collection with this property, then $|\C| = k(n-k)+1$.
\end{conjecture}

We will present a stronger statement, which implies both these conjectures.

In \cite{Post}, a stratification of the totally 
nonnegative Grassmannian $Gr^{tnn}(k,n)$, called the {\it positroid stratification,}
was investigated.
%Postnikov \cite{Post} stratified the Grassmannian $G(k,n)$ into finitely many pieces.
The cells of this stratification can be labeled by many different kinds of combinatorial objects.
The three of these objects which we will use are
\begin{enumerate}
 \item Grassmann necklaces, which are certain sequences $(I_1, I_2, \ldots, I_n)$ in $\binom{[n]}{k}$.
\item Positroids, which are certain subsets of $\binom{[n]}{k}$.
\item Decorated permutations, which are permutations of $[n]$ with fixed points colored in two colors.
\end{enumerate}
There is a length function $\ell$ on decorated permutations, related to the length function on the affine symmetric group.

One of our main results is the following.
\begin{theorem}
\label{thm:main_intro}
Let $\I$, $\M$ and $\pi$ be a Grassmann necklace, positroid and decorated
permutation  corresponding to each other.
Let $\C \subseteq \binom{[n]}{k}$ be a collection of pairwise weakly separated
sets such that $\I \subseteq \C \subseteq \M$, and such that $\C$ is not
contained in any larger collection with this property. Then $|\C| =
1+\ell(\pi)$.  
\end{theorem}

We will call collections $\C$ satisfying the conditions of this theorem
{\it maximal weakly separated collections inside $\M$.}

This result implies Conjectures~\ref{conj:LZ} and~\ref{conj:Scott}, and also
the $w$-chamber conjecture of Leclerc and Zelevinsky \cite{LZ}.

From our perspective, Scott's Conjecture~\ref{conj:Scott} is especially
natural, because it refers to the largest cell in $G^{tnn}(k,n)$.  
The reader may find it useful to first consider the proofs of the main results in that special case.

We also prove the following result on mutation-connectedness, whose special cases were
conjectured in \cite{LZ} and \cite{Scott}.

\begin{theorem}
\label{thm:mutations_intro}  
Fix a positroid $\M$.
Any two maximal weakly separated collections inside $\M$ can be obtained from each other 
by a sequence of mutations of the following form:
$\C\mapsto (\C \setminus \{ S a c \}) \cup \{ S b d \}$,
assuming that, for some cyclically ordered elements $a, b, c, d$ in $[n] \setminus S$,
$\C$ contains $S a b$, $S b c $, $S c d$, $S d a$ and $S a c$.
(Here $S a b$ is a shorthand for $S\cup\{a,b\}$, etc.)
\end{theorem}

Our main tool is the technology of plabic graphs, developed in \cite{Post}. 
Plabic graphs are variants of wiring diagrams of decompositions 
in the symmetric group, which deserve to be better understood.
As with wiring diagrams, there is a notion of a reduced plabic graph, 
and the faces of a reduced plabic graph are labeled with subsets of $\binom{[n]}{k}$.  

Actually, we show that weakly separated collections and reduced plabic graphs
are in a bijective correspondence.

\begin{theorem}
\label{th:intro_correspondence}
Fix a positroid $\M$ and the corresponding Grassmann necklace $\I$.
For a reduced plabic graph $G$ associated with $\M$, let $\f(G) \subset \binom{[n]}{k}$ 
be the collection of labels of faces of $G$.  Then the map $G\mapsto \f(G)$ is a bijection
between reduced plabic graphs for the positroid $\M$ and maximal
weakly separated collections $\C$ such that $\I \subseteq \C \subseteq \M$.
\end{theorem}

We also establish the following consequence for the theory of cluster algebras, in Section~\ref{sec:Clust}.
\begin{theorem} 
Let $\C$ be a subset of $\binom{[n]}{k}$. The following are equivalent:
\begin{enumerate}
\item The set of Pl\"ucker coordinates $\{ p_I \}_{I \in \C}$ is a cluster in the cluster algebra structure on the coordinate ring of $G(k,n)$.
\item $\C$ is a maximal weakly separated collection.
\item $\C$ is the collection of face labels of a reduced plabic graph for the uniform matroid $\binom{[n]}{k}$.
\end{enumerate}
\end{theorem}
Readers not interested in cluster algebras may skip this section.

%We show (see REFERENCE for precise statements) that maximal weakly separated collections correspond
%are the labels of the regions of reduced alternating strand diagrams.  

Theorems~\ref{thm:main_intro} and~\ref{thm:mutations_intro}  
follow from the correspondence in 
Theorem~\ref{th:intro_correspondence} and the properties of plabic graphs from \cite{Post}.

%This allows us to use techniques from~\cite{Post}.  

Our main new tool is a construction we term ``plabic tilings". Maximal plabic tilings are dual to plabic graphs, and non-maximal
plabic tilings correspond to non-maximal weakly separated collections.

\subsection{Parallel work}

As we were completing this work, we learned that Danilov, Karzanov and Koshevoy succeeded
in proving Leclerc and Zelevinsky's conjectures. Like our proof, their proof
relies on certain planar diagrams. Their diagrams, which are called generalized
tilings, are closely related to the plabic graphs which play a central role in
this paper, see \cite{DKK}. We see several differences between our
work and theirs. We study weakly separated collections of all positroids, which is
more general than considering weakly separated collections of $w$-chamber
sets, see Section \ref{sec:Connection_with_LZ}.
Moreover, our constructions respect
the dihedral symmetry inherent in the definitions of positroids and weak
separation, which Danilov, Karzanov and Koshevoy's constructions do not. Finally, our
bicolored surfaces give a natural object to assign to a weakly separated
collection which is not maximal. 
%As we point out in example REFERENCE, this is
%extremely helpful for hand computations, and we hope it will be of theoretical
%use as well.

\subsection{Notion of weak separation}

From our perspective, it is most natural to work with weak separation as defined by Scott \cite{Scott}, which is motivated by geometric considerations in the case of Grassmannians. 
A similar notion was introduced by Leclerc and Zelevinsky, \cite{LZ},  motivated by considerations in the study of flag manifolds. 
We will refer to Leclerc and Zelevinsky's definition as \newword{LZ weak separation}.
As we explain in section~\ref{sec:Connection_with_LZ}, LZ weak separation is actually a special case of weak separation.
All of the statements in section~\ref{sec:Connection_with_LZ} (although not all their proofs) can be understood as soon as the reader has read the definitional material in sections~\ref{sec:WS} and~\ref{sec:Grneck}.

\subsection{Acknowledgments}

The third author is engaged in an as yet unpublished project with
Andre Henriques and Dylan Thurston, which has some overlap with this work.  In
particular, Proposition~\ref{prop:embed} was originally proved in that
collaboration; the third author is grateful to his collaborators for letting it
appear here. 

%The second author was supported in part by NSF CAREER Award DMS-0504629; the
%third author was supported in part by a Clay Research Fellowship.

% WHAT ELSE?

\section{Notation}

Throughout the paper we use the following notation.
Let $[n]$ denote $\{ 1,\ldots, n \}$, and $\binom{[n]}{k}$ denote the set of $k$ element subsets of $[n]$.
We will generally consider $[n]$ as cyclically ordered.  
We write $<_{i}$ for the cyclically shifted linear order  on $[n]$:
$$i <_{i} i+1 <_{i} i+2 <_{i} \cdots <_{i} n <_{i} 1 <_{i} \cdots <_{i} i-1.$$ 
We will say that
$i_1$, $i_2$, \ldots, $i_r$ in $[n]$ are {\it cyclically ordered\/} 
if $i_1 <_i i_{2} <_i \cdots <_i i_r$ for some $i\in[n]$.

We write $(a, b)$ for the open cyclic interval from $a$ to $b$. In other words,
the set of $i$ such that $a$, $i$, $b$ is cyclically ordered.  We write $[a,
b]$ for the closed cyclic interval, $[a, b] = (a, b) \cup \{ a, b \}$, and use
similar notations for half open intervals.

If $S$ is a subset of $[n]$ and $a$ an
element of $[n]$, we may abbreviate $S \cup \{ a \}$ and $S \setminus \{ a \}$
by $S a$ and $S \setminus a$. 

In this paper, we need to deal with three levels of objects: elements of $[n]$,
subsets of $[n]$, and collections of subsets of $[n]$. For clarity, we will
denote these by lower case letters, capital letters, and calligraphic letters,
respectively. 

The use of the notation $I \setminus J$ does not imply $I \subseteq J$.

\section{Weakly separated collections}
\label{sec:WS}

In this section, we define weak separation for collections of $k$ element subsets and
discuss the $k$ subset analogue of Leclerc and Zelevinsky's conjectures.  The
relation of this approach to the original definitions and conjectures from
\cite{LZ} will be discussed in section~\ref{sec:Connection_with_LZ}.

Let us fix two nonnegative integers $k \leq n$.
%Let $[n]:= \{ 1,\ldots, n \}$ and let $\binom{[n]}{k}$ denotes the set of $k$ element subsets of $[n]$.

%We will generally consider $[n]$ as cyclically ordered.  We will say that
%$i_1$, $i_2$, \ldots, $i_r$ in $[n]$ are {\it cyclically ordered\/} 
%if $i_s < i_{s+1} < \cdots < i_r < i_1 < i_2 < \cdots < i_{s-1}$ for some $s\in[r]$. 
%
%In this paper, we need to deal with three levels of objects: elements of $[n]$,
%subsets of $[n]$, and collections of subsets of $[n]$. For clarity, we will
%denote these by lower case letters, capital letters, and calligraphic letters,
%respectively. 

\begin{definition}
%Let $[n]:=\{ 1,2, \ldots, n\}$, and let $\binom{[n]}{k}$ be the set of $k$ element subsets of $[n]$.
For two $k$ element subsets $I$ and $J$ of $[n]$, we say that 
%Let $I$ and $J$ be two $k$ element subsets of $[n]$. 
%Let $A = I \setminus J$ and $B = J \setminus I$. 
$I$ and $J$ are {\it weakly separated\/}
%Define $I$ and $J$ to be \newword{weakly separated} 
if there do \emph{not} exist $a$, $b$, $a'$, $b'$, cyclically ordered, with $a, \, a' \in I\setminus J$ and 
$b,\, b' \in J\setminus I$. 

Geometrically, $I$ and $J$ are weakly separated if and only if there exists a chord separating
the sets $I\setminus J$ and $J\setminus I$ drawn on a circle.

%, then $I$ and $J$ are weakly separated if we can draw a chord separating $A$ from $B$. 

We write $I \parallel J$ to indicate that $I$ and $J$ are weakly separated.

We call a subset $\C$ of $\binom{[n]}{k}$ a {\it collection.} 
We define a {\it weakly separated collection\/} to be a collection 
$\C\subset \binom{[n]}{k}$ such that, for any $I$ and $J$ in $\C$, 
the sets $I$ and $J$ are weakly separated. 

We define a {\it maximal weakly separated collection} to be a weakly separated collection which is not contained 
in any other weakly separated collection.  
\end{definition}

Following Leclerc and Zelevinsky \cite{LZ}, Scott observed the following claim.

\begin{proposition}
\cite{Scott}, cf.\ \cite{LZ}
\label{mutation}
Let $S \in \binom{[n]}{k-2}$ and let $a$, $b$, $c$, $d$ be cyclically ordered elements of $[n] \setminus S$. 
Suppose that a maximal weakly separated collection $\C$ contains $S a b$, $S b c$, $S c d$, $S d a$ and $S a c$. 
Then $\C' := (\C \setminus \{ S a c \}) \cup \{ S b d \}$ is also a maximal weakly separated collection.
\end{proposition}

We define $\C$ and $\C'$ to be {\it mutations} of each other if they are linked as in Proposition~\ref{mutation}.
We will prove the following claim, conjectured by Scott.

\begin{theorem} \label{SConjecture}
Every maximal weakly separated collection of $\binom{[n]}{k}$ has cardinality $k(n-k)+1$. 
Any two maximal weakly separated collections are linked by a sequence of mutations.
\end{theorem}

\section{Weakly separated collections in positroids}
\label{sec:Grneck}

While our results are purely combinatorial, they are motivated by constructions in algebraic geometry.
Specifically, \cite{Post} introduced the positroid stratification of the Grassmannian (see also \cite{KLS}). 
We will prove a version of Theorem~\ref{SConjecture} for every cell in this stratification. 
Theorem~\ref{SConjecture} itself will correspond to the case of the largest
cell.  

There are several combinatorial objects which can be used to index the cells of this stratification. 
We will use three of these  -- {\it Grassmann necklaces, decorated permutations,} 
and {\it positroids}.  See \cite{Post} for more details. 

%Recall that $<_{i}$ is the cyclically shifted linear order on $[n]$:
%$i <_{i} i+1 <_{i} i+2 <_{i} \cdots <_{i} n <_{i} 1 <_{i} \cdots <_{i} i-1.$ 

\begin{definition} \cite[Definition~16.1]{Post}
A \textit{Grassmann necklace} is a sequence $\I = (I_1, \cdots, I_n)$ of $k$ element subsets of $[n]$ such that, for $i \in [n]$, the set $I_{i+1}$ contains $I_i \setminus \{ i \}$.  (Here the indices are taken modulo $n$.)
If $i \not \in I_{i}$, then we should have $I_{i+1} = I_i$. 

In other words, $I_{i+1}$ is obtained from $I_i$ by deleting $i$ and adding another element,
or $I_{i+1} = I_i$.  Note that, in the latter case $I_{i+1}=I_i$,  either $i$ belongs to all elements $I_j$
of the Grassmann necklace, or $i$ does not belong to all elements $I_j$ of the necklace.
\end{definition}

Here is an example of a Grassmann necklace: $I_1 = \{1,2,4\}$, $I_2 =
\{2,4,5\}$, $I_3 = \{3,4,5\}$, $I_4 = \{4,5,2\}$, $I_5 = \{5,1,2\}$.

%The corresponding Grassmann necklace is $I_1 = \{1,2,3,6\},I_2 =
%\{2,3,6,8\},I_3 = \{3,6,8,1\},I_4 = \{4,6,8,1\},I_5 = \{6,8,1,2\}, I_6 =
%\{6,8,1,2\}, I_7 = \{7,8,1,2\}, I_8 = \{8,1,2,3\}$.

% There is a certain sense in which the colors $1$ and $-1$ should be thought
% of as corresponding to sending $i$ to itself and sending $i$ to $i +n \mod
% n$, respectively. For more on this perspective, see~\cite{KST}.

Recall the linear order $<_i$ on $[n]$.
We extend  $<_i$ to $k$ element sets, as follows.
For $I=\{i_1, \cdots, i_k \}$ and $J=\{j_1, \cdots, j_k \}$ with 
$i_1 <_i i_2 \cdots <_i i_k$ and $j_1 <_i j_2 \cdots <_i j_k$, define the partial order
$$
I \leq_i J 
\textrm{ if and only if } i_1 \leq_i j_1, \cdots, i_k \leq_i j_k.
$$
In other words, $\leq_i$ is the cyclically shifted termwise partial order on $\binom{[n]}{k}$.

%Let $t \in [n]$ and $I$, $J \in \binom{[n]}{k}$ with 
%$$J=\{j_1, \cdots, j_k \}, \ j_1 <_t j_2 \cdots <_t j_k.$$
%We define a partial order

\begin{definition}
Given a Grassmann necklace $\I=(I_1,\cdots,I_n)$, define the {\it positroid} $\M_{\I}$ to be
$$
\M_{\I} := \{J \in \binom{[n]}{k} \mid I_i \leq_i J \text{ for all } i \in [n] \}.
$$
This turns out to be a special kind of matroid.  
\end{definition}

The term ``positroid'' is an abbreviation for ``totally positive matroid''.
These are exactly matroids that can be represented by totally positive
matrices, see \cite{Post}.

\begin{definition}
Fix a Grassmann necklace $\I=(I_1,\cdots,I_n)$, with corresponding positroid
$\M_{\I}$. Then $\C$ is called a {\it weakly separated collection inside}
$\M_{\I}$ if $\C$ is a weakly separated collection and $\I \subseteq \C
\subseteq \M_{\I}$. We call $\C$ a {\it maximal weakly separated collection
inside} $\M_{\I}$ if it is maximal among weakly separated collections inside
$\M_{\I}$.  
\end{definition}

We have not shown yet that $\I \subseteq \M_{\I}$, or that $\I$ is weakly
separated, so there is a risk that there are no weakly separated collections in
$\M_{\I}$.  We remedy this by the following lemma.

%in Proposition~\ref{p:IinM}.
%We now redeem an earlier promise:

\begin{lemma} \label{lem:IinM}
For any Grassmann necklace $\I$, we have $\I \subseteq \M_{\I}$, and $\I$ is weakly separated.
\end{lemma}

\begin{proof}
For every $i$ and $j$ in $[n]$, we must show that $\I_i \leq_i \I_j$ and $\I_i \parallel \I_j$.

By the definition, $I_{r+1}$ is either obtained from $I_r$ by deleting $r$ and adding another element,
or else $I_{r+1} = I_r$.
As we do the changes $I_1 \to I_2 \to \cdots \to I_n\to I_1$, we delete each $r\in[n]$ at most once 
(in the transformation $I_r\to I_{r+1}$).  This implies that we add each $r$ at most once.

Let us show that $I_j\setminus I_i\subseteq [j,i)$.  
Suppose that this is not true and there exists $r\in (I_j\setminus I_i)\cap [i,j)$.
(Note that $I_{r+1} \ne I_r$.  Otherwise, 
$r$ belongs to all elements of the Grassmann necklace, or $r$ does not belong to all elements of the necklace.)
Consider the sequence of changes 
$I_i \to I_{i+1} \to \cdots \to I_r\to I_{r+1}\to \cdots \to I_j$.  
We should have $r\not\in I_i$, $r\in I_r$, $r\not\in I_{r+1}$, $r\in I_j$.
Thus $r$ should be added twice, as we go from $I_i$ to $I_r$ and 
as we go from $I_{r+1}$ to $I_j$.  We get a contradiction.

Thus $I_j\setminus I_i\subseteq [j,i)$ and, similarly,
$I_i\setminus I_j \subseteq [i,j)$.
We conclude that $\I_i \leq_i \I_j$ and $\I_i \parallel \I_j$, as
desired.
\end{proof}

Our main theorem will say, in part, that all maximal weakly separated collections inside $\M_{\I}$ have the same size. 
To describe this cardinality, we define decorated permutations.

\begin{definition} \cite[Definition~13.3]{Post} 
A decorated permutation $\pi^{:} = (\pi, col)$ is a permutation $\pi \in S_n$ together 
with a coloring function $col$ from the set of fixed points $\{i \mid \pi(i) = i\}$ to $\{1,-1\}$. \end{definition}

There is a simple bijection between decorated permutations and
Grassmann necklaces.
% (for same $n$ and all $k=0,1,\dots,n$.)
To go from a Grassmann necklace $\I$ to a decorated permutation
$\pi^{:}=(\pi,col)$, we set $\pi(i) = j$ whenever $I_{i+1} =
(I_{i}\setminus\{i\}) \cup\{j\}$ for $i\ne j$.  If $i \not\in I_{i} = I_{i+1}$
then $\pi(i)=i$ is a fixed point of color $col(i)=1$.  Finally, if $i\in
I_i = I_{i+1}$ then $\pi(i)=i$ is a fixed point of color $col(i)=-1$.

To go from a decorated permutation $\pi^{:}=(\pi,col)$ to a Grassmann necklace $\I$, we set
$$
I_i = \{ j \in [n] \mid j <_i \pi^{-1}(j) \textrm{ or } (\pi(j)=j \textrm{ and } col(j)=-1) \}.
$$

For example, the decorated permutation $\pi^:=(\pi,col)$ with $\pi = 81425736$ and $col(5)=1$ corresponds
to the Grassmann necklace $(I_1,\dots,I_8)$ with
$I_1 = \{1,2,3,6\}$, $I_2 = \{2,3,6,8\}$, $I_3=\{3,6,8,1\}$, $I_4=\{4,6,8,1\}$,
$I_5= \{6,8,1,2\}$, $I_6 = \{6,8,1,2\}$, $I_7= \{7,8,1,2\}$, $I_8=\{8,1,2,3\}$.
For the fact that this is a bijection between Grassmann necklaces and decorated permutations, see~\cite[Theorem 17.1]{Post}.

%There is some invariant associated to a Grassmann necklace called the
%\newword{length} of $\I$. It is denoted by $l(\I)$ and this number will be
%used to express the cardinality of the maximal weakly separated collection of
%$\M_{\I}$. But to define this, it is easier to think of Grassmann necklaces as
%\newword{decorated permutations}. These are combinatorial objects defined in
%\cite{postnikov-2006} that are in bijection with the Grassmann necklaces.

\begin{definition} \cite[Section~17]{Post}
For $i$, $j \in [n]$,  we say that $\{i,j\}$ forms an {\it alignment\/} in $\pi$ if
$i$, $\pi(i)$, $\pi(j)$, $j$ are cyclically ordered (and all distinct). The {\it length\/} $\ell(\pi^{:})$
is defined to be $k(n-k)-A(\pi)$ where $A(\pi)$ is the number of alignments in
$\pi$. We define $\ell(\I)$ to be $\ell(\pi^{:})$ where $\pi^{:}$ is the
associated decorated permutation of $\I$.  
\end{definition}

%Grassmann necklace and decorated permutations will be heavily used in the next
%section, where we will define plabic graphs and each plabic graph will be
%associated to a Grassmann necklace. Since we have defined $l(\I)$, now we are
%ready to state the generalized version of Theorem~\ref{Sconjecture}.

We now state our result for an arbitrary Grassmann cell.

\begin{theorem}
\label{thm:GSConjecture}
Fix any Grassmann necklace $\I$. Every maximal weakly separated collection inside $\M_{\I}$ has cardinality $\ell(\I)+1$. Any two maximal weakly separated collections inside $\M_{\I}$ are linked by a sequence of mutations.
\end{theorem}

As a particular case, let $I_i = \{i,i+1,\dots,i+k-1\}\subset [n]$. (The entries of $I_i$ are taken modulo $n$.) 
Then $\M_{\I} = \binom{[n]}{k}$ and the above theorem becomes Theorem~\ref{SConjecture}. In Section~\ref{sec:Connection_with_LZ}, we explain how this theorem also incorporates the conjectures of Leclerc and Zelevinsky.

\section{Decomposition into connected components}  \label{Connected}

Let $\I$ be a Grassmann necklace, let $\pi^:$ be the corresponding decorated
permutation, and let $\M$ be the corresponding positroid $\M_{\I}$. 

The connected components of $\pi^:$, $\I$, and $\M$ are certain decorated permutations,
Grassmann necklaces, and positroids, whose ground sets may no longer be $[n]$ but rather
subsets of $[n]$, which inherit their circular order from $[n]$.  
%(Naturally, these objects can be defined on any cyclically
%ordered ground set.  We picked the ground set $[n]$ only for simplicity of notation.)

\begin{example}
The arrows in the left hand side of Figure~\ref{fig:Disconnected} form a decorated permutation with $(k,n)=(5,10)$.
The connected components are $019$, $2378$ and $456$. 
The $5$ element set drawn on the faces of this graph form a weakly separated collection for this $\pi$.
The reader who is curious about the bipartite graph in Figure~\ref{fig:Disconnected}, and about the right hand side of the figure, will have her curiosity satisfied in Sections~\ref{sec:PlabicGraphs} and~\ref{sec:tilings}.
\end{example}

\begin{definition}
Let $[n] = S_1 \sqcup S_2 \sqcup \cdots \sqcup S_r$ be a partition of $[n]$ into disjoint subsets. We say that $[n]$ is \newword{noncrossing} if, for any circularly ordered $(a,b,c,d)$, we have $\{ a,c \} \subseteq S_i$ and $\{ b,d \} \subseteq S_j$ then $i=j$. 
\end{definition}

See, for example,~\cite{Simion} for background on noncrossing partitions.
The following is obvious from the definition:
\begin{prop}
The common refinement of two noncrossing partitions is noncrossing.
\end{prop}
Therefore, the following definition makes sense.

\begin{definition}
Let $\pi^{\colon}$ be a decorated permutation. Let $[n] = \bigsqcup S_i$ be the finest noncrossing partition of $[n]$ such that, if $i \in S_j$ then $\pi(i) \in S_j$. 

Let $\pi^:_{(j)}$ be the restriction of $\pi^:$ to the set $S_j$,
and let $\I_{(j)}$ be the associated Grassmann necklace on the ground set $S_j$, 
for $j=1,\dots,r$. 

We call $\pi^:_{(j)}$ the {\it connected components\/} of $\pi^:$,
and $\I_{(j)}$ the {\it connected components\/} of $\I$. 

We say that $\pi^:$ and $\I$ are connected if they have exactly one connected component.
\end{definition}

\begin{lemma} \label{lem:TwoIntervals}
The decorated permutation $\pi^{\colon}$ is disconnected if and only if there are two circular intervals $[i,j)$ and $[j,i)$ such that $\pi$ takes $[i,j)$ and $[j,i)$ to themselves.
\end{lemma}

\begin{proof}
If such intervals exist, then the pair $[n] = [i,j) \sqcup [j,i)$ is a noncrossing partition preserved by $\pi$. So there is a nontrivial noncrossing partition preserved by $\pi$ and $\pi$ is not connected. 
Conversely, any nontrivial noncrossing permutation can be coarsened to a pair of intervals of this form so, if $\pi$ is disconnected, then there is a pair of intervals of this form.
\end{proof}

Note that each fixed point of $\pi^:$ (of either color) forms a connected component.

\begin{lemma}
A Grassmann necklace $\I=(I_1,\dots,I_n)$ is connected if and only if the sets $I_1,\dots, I_n$ are all distinct. 
\end{lemma}

\begin{proof}
If $\pi^{\colon}$ is disconnected then let $[i,j)$ and $[j,i)$ be as in Lemma~\ref{lem:TwoIntervals}.
As we change from $I_i$ to $I_{i+1}$ to $I_{i+2}$ to \dots to $I_j$, each element of $[i,j)$ is removed once and is added back in once. So $I_i=I_j$. 

Conversely, suppose that $I_i=I_j$. As we change from $I_i$ to $I_{i+1}$ to $I_{i+2}$ and so forth, up to $I_j$, each element of $[i,j)$ is removed once. In order to have $I_i=I_j$, each element of $[i,j)$ must be added back in once. So $\pi$ takes $[i,j)$ to itself.
\end{proof}

In this section, we will explain how to reduce computations about positroids to the connected case.

So, for the rest of this section, suppose that $I_i=I_j$ for some $i \neq j$.
Set $I^1 := [i,j) \cap I_i$ and $I^2 = [j,i) \cap I_i$; set $k_1 = |I^1|$ and $k_2 = |I^2|$.
We will also write $n^1 = |[i,j)|$ and $n^2 = |[j,i)|$. 

\begin{proposition} \label{DirectSum}
For every $J \in \M$, we have $|J \cap [i,j)| = k^1$ and $|(J \cap [j,i)| = k^2$. 
\end{proposition}

\begin{proof} 
Since $I_i$ is the $\leq_i$ minimal element of $\M$, we have 
$$| [i,j) \cap J| \leq | [i,j) \cap I_i |=k^1$$
for all $J \in \M$. But also, similarly,
$$| [j,i) \cap J| \leq | [j,i) \cap I_j | =k^2 $$
Adding these inequalities together, we see that
$$|J| = |[i,j) \cap J| + | [j,i) \cap J| \leq k^1 + k^2 = k.$$

But, in fact, $|J| =k$. So we have equality at every step of the process.
In particular,  $|J \cap [i,j)| = k^1$ and $|J \cap [j,i)| = k^2$.
\end{proof}

\begin{proposition} \label{p:ItsDirectSum}
The matroid $\M$ is a direct sum of two matroids $\M^1$ and $\M^2$, supported on the ground sets  $[i,j)$ and $[j,i)$, having ranks $k^1$ and $k^2$.
In other words, there are matroids $\M^1$ and $\M^2$ such that $J$ is in $\M$ if and only if $J \cap [i,j)$ is in $\M^1$ and $J \cap [j,i)$ is in $\M^2$.
\end{proposition}

For every $J \in \M$, write $J^1 := J \cap [i,j)$ and $J^2 := J \cap [j,i)$.

\begin{proof}
This is an immediate corollary of Proposition~\ref{DirectSum}; see~\cite[Theorem~7.6.4]{Bry} for the fact that this rank condition implies that $\M$ is the direct sum of $\M|_{[i,j)}$ and $\M|_{[j,i)}$.
See \cite[Proposition~7.6.1]{Bry} for the interpretation of this in terms of bases of $\M$.
\end{proof}

\begin{proposition}
For $k \in [i,j]$, the set $I_k$ is of the form $J \cup I^2$ for some $J \in \M^1$. 
For $k \in [j,i]$, the set $I_k$ is of the form $I^1 \cup J$ for some $J \in \M^2$. 
\end{proposition}

\begin{proof}
Consider the case that $k \in [i,j]$, the other case is similar.
Recall that $I_k$ is the $\leq_k$ minimal element of $\M$.
Since $\M = \M^1 \oplus \M^2$, we know that $I_k = J^1 \cup J^2$, where $J^r$ is the $\leq_k$ minimal element of $\M^1$.
But, on $[j,i)$, the orders $\leq_i$ and $\leq_k$ coincide, so $J^2$ is the $\leq_i$ minimal element of $\M^2$, namely $I^2$.
\end{proof}

View $[i,j)$ as circularly ordered. 
Let $\I^1$ denote the circularly ordered sequence $(I_i^1, I_{i+1}^1, \ldots, I_{j-1}^1)$.

\begin{proposition}
$\I^1$ is a Grassmann necklace on the ground set $[i,j)$, and $\M^1$ is the associated positroid.
\end{proposition}

\begin{proof}
For $k \in [i,j)$, we know that $I_{k+1} \supseteq I_k \setminus \{ k \}$. 
Since $I_{k} = I^1_k \cup I^2$, and $I_{k+1} = I^1_{k+1} \cup I^2$, we have $I^1_{k+1} \supseteq I^1_k \setminus \{ k \}$. 
This is the definition of a Grassmann necklace. 
(Note that we have used the condition $I_i=I_j$ to cover the boundary case $k=j-1$.)

Now, we show that $\M$ is the associated positroid.
Consider any $J \in \binom{[i,j)}{k^1}$.
If $J \in M^1$ then $J \cup I^2 \in M$, so $J \cup I^2 \geq_k I_k$ for every $k \in [n]$.
This immediately implies that  $J \geq_k I^1_k$ for every $k \in [i,j)$, so $J$ is in the positroid $\M_{\I^1}$.

Conversely, suppose that $J$ is in the positroid $\M_{\I^1}$.
We wish to show that $J \cup I^2$ is in $\M$.
Reversing the argument of the previous paragraph shows that $J \cup I^2 \geq_k I_k$ for all $k \in [i,j)$.
For $k \in [j,i)$, we know that $I^2 \geq_k I_k^2$ and, since $\geq_k$ and $\geq_i$ coincide on $[i,j)$, we know that $J \geq_k I^1$. So $J \cup I^2 \geq_k I_k$ for $k$ in $[j,i)$ as well.
So $J \in \M$.
\end{proof}

It is easy to check the following:
\begin{proposition}
With the above definitions, $\ell(\I) =\ell(\I^1) + \ell(\I^2)$.
\end{proposition} 

We now study weakly separated collections in $\M$. 

\begin{lemma} \label{SplitLemma}
Let $J = J^1 \cup J^2 \in \M$. If $J$ is weakly separated from $I^1 \cup I^2$, then either $I^1=J^1$ or $I^2 = J^2$.
\end{lemma}

\begin{proof}
Suppose, to the contrary, that $J^1 \neq I^1$ and $J^2 \neq I^2$. 
Since $I^1 <_i J^1$, there are $a$ and $b \in [i,j)$, with $i \leq_i a <_i b$, such that $a \in I^1 \setminus J^1$ and $b \in J^1 \setminus I^1$. 
Similarly, there are $c$ and $d \in [j,i)$, with $j \leq_j c <_j d$, such that $c \in I^2 \setminus J^2$ and $d \in J^2 \setminus I^2$. 
Then $a$ and $c$ are in $I^1 \cup I^2 \setminus J^1 \cup J^2$, while $b$ and $d$ are in $J^1 \cup J^2 \setminus I^1 \cup I^2$.
So $I^1 \cup I^2$ and $J$ are not weakly separated.
\end{proof}

\begin{proposition} \label{SplittingSummary}
If $\C$ is a weakly separated collection in $\M$, then there are weakly separated collections $\C^1$ and $\C^2$ in $\M^1$ and $\M^2$ such that
$$\C = \{ J \cup  I^2 : J \in \C^1 \} \cup \{ I^1 \cup  J : J \in \C^2 \} .$$
Conversely, if $\C^1$ and $\C^2$ are weakly separated collections in $\M^1$ and $\M^2$, then the above formula defines a weakly separated collection in $\M$.
The collection $\C$ is maximal if and only if $\C^1$ and $\C^2$ are.
\end{proposition}

\begin{proof}
First, suppose that $\C$ is a weakly separated collection in $\M$. Since $\I \subset \C$, we have $I^1 \cup I^2 \in \C$. 
By Lemma~\ref{SplitLemma}, every $J \in \C$ is either of the form $J^1 \cup I^2$, or $I^1 \cup J^2$. 
Let $\C^r$ be the collection of all sets $J^r$ for which $J^r \cup I^{3-r}$ is in $\C$.
The condition that $\C$ is weakly separated implies that $\C^r$ is; the condition that $\I \subseteq \C \subseteq \M$ implies that $\I^r \subseteq \C^r \subseteq \M^r$. 
So $\C^r$ is a weakly separated collection in $\M^r$ and it is clear that $\C$ is built from $\C^1$ and $\C^2$ in the indicated manner.

Conversely, it is easy to check that, if $\C^1$ and $\C^2$ are weakly separated collections in $\M^1$ and $\M^2$, then the above formula gives a weakly separated collection in $\M$. 

Finally, if $\C \subsetneq \C'$ with $\C'$ a weakly separated collection in $\M$, then either $\C^1 \subsetneq (\C')^1$ or $\C^2 \subsetneq (\C')^2$. So, if $\C$ is not maximal, either $\C^1$ or $\C^2$ is not. The converse is similar.
\end{proof}

In summary, the Grassmann necklace $\I$ can be described in terms of smaller necklaces $\I^1$ and $\I^2$; 
the positroid $\M$ can be described in terms of smaller positroids $\M^1$ and $\M^2$;
weakly separated collections $\C$ in $\M$ can be described in terms of weakly separated collections in $\M^1$ and $\M^2$.

\begin{remark}
The connected components $S_i$ are the connected components of $\M$ in the sense of matroid theory. We have partially proved this in this section. In order to fully establish this result we would also need to show that, if $\pi^{\colon}$ is connected, then the matroid $\M$ is connected. We don't need this fact, but it is true.
\end{remark}

\section{Plabic graphs} \label{sec:PlabicGraphs}

Plabic graphs were introduced in \cite{Post}. 
We will use these graphs to represent weakly separated collections and the properties of these graphs will be a major key in our proof of the conjecture. For more details on plabic graphs, see \cite{Post}.

\begin{definition}
A \newword{planar bicolored graph}, or simply a \newword{plabic graph} is a planar undirected graph $G$ drawn inside a disk. The vertices on the boundary are called boundary vertices, and are labeled in clockwise order by $[n]$. All vertices in the graph are colored either white or black.
\end{definition}

%NEW MATERIAL HERE

Let $G$ be a plabic graph in the disc $D$. We will draw $n$ directed paths, called \newword{strands}, within the disc $D$, each starting from and ending at a boundary vertex of $G$. 

\begin{definition}
The construction in this definition is depicted in Figure~\ref{fig:pla_label}; the numeric labels in that figure will be explained below. The strands are drawn as follows: For each edge of $G$,  draw two strand segments. If the ends of the segment are the same color, then the two strands should be parallel to the edge without crossing, and should run in opposite directions. If the two ends are different colors, then the two strands should cross, with one running towards each endpoint. As we will discuss in the below remark, for most purposes, we can reduce to the case that $G$ is bipartite, so the latter case is the important one.  Around each vertex, connect up the ends of the strands so that they turn right at each black vertex and left at each white vertex. We will have $n$ strands leading from $\partial G$ to itself, plus possibly some loops in the interior of $G$. 
\end{definition}

\begin{figure}
\centerline{\scalebox{0.5}{
\psfrag{123}{\Huge $123$}
\psfrag{234}{\Huge $234$}
\psfrag{345}{\Huge $345$}
\psfrag{456}{\Huge $456$}
\psfrag{567}{\Huge $567$}
\psfrag{678}{\Huge $678$}
\psfrag{178}{\Huge $178$}
\psfrag{128}{\Huge $128$}
\psfrag{127}{\Huge $127$}
\psfrag{137}{\Huge $137$}
\psfrag{136}{\Huge $136$}
\psfrag{135}{\Huge $135$}
\psfrag{134}{\Huge $134$}
\psfrag{167}{\Huge $167$}
\psfrag{156}{\Huge $156$}
\psfrag{145}{\Huge $145$}
\psfrag{1}{\Huge $1$}
\psfrag{2}{\Huge $2$}
\psfrag{3}{\Huge $3$}
\psfrag{4}{\Huge $4$}
\psfrag{5}{\Huge $5$}
\psfrag{6}{\Huge $6$}
\psfrag{7}{\Huge $7$}
\psfrag{8}{\Huge $8$}
\includegraphics{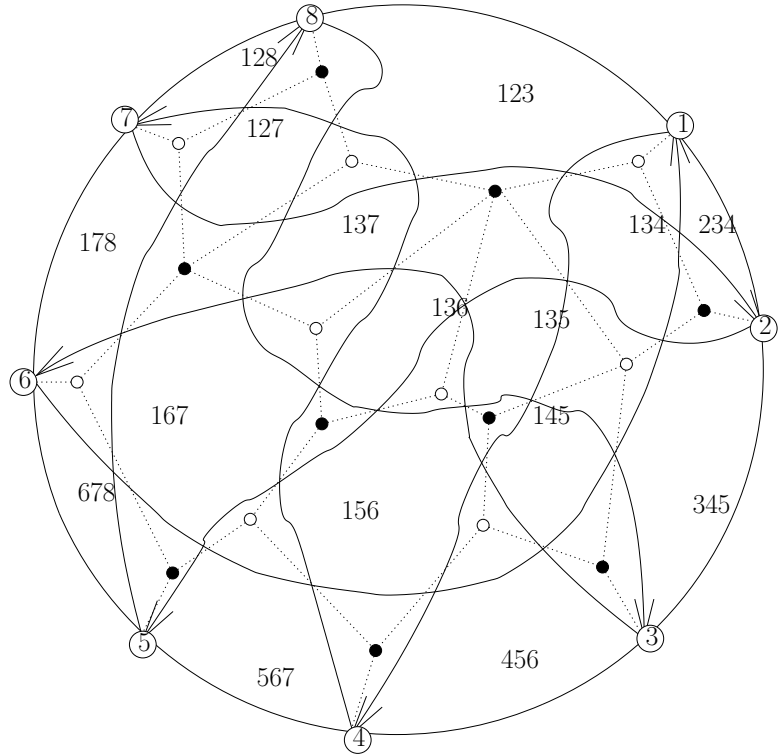}
}}	
	\caption{Labeling faces of a plabic graph}
	\label{fig:pla_label}
\end{figure}

\begin{remark}\label{rem:Contract}
Suppose that $u$ and $v$ are two vertices of $G$, joined by an edge, which have the same color.
Let $G'$ be the graph formed by contracting the edge $(u,v)$ to a single vertex $w$, and coloring $w$ the same color as $u$ and $v$ are colored. Then the strands of $G$ and $G'$ have the same connectivity. See Figure~\ref{Contraction} for a depiction of how contracting an edge leaves strand topology unchanged. For this reason, we can usually reduce any question of interest to the case where $G$ is bipartite. We do not restrict in this paper to bipartite graphs for two reasons: (1) We want to be compatible with the definitions in~\cite{Post}, which does not make this restriction and (2) The description of the square move, (M1) below, would be significantly more complicated.
\end{remark}

\begin{figure}
\centerline{\scalebox{0.5}{
\psfrag{G}{\Huge $G$}
\psfrag{H}{\Huge $G'$}
\includegraphics{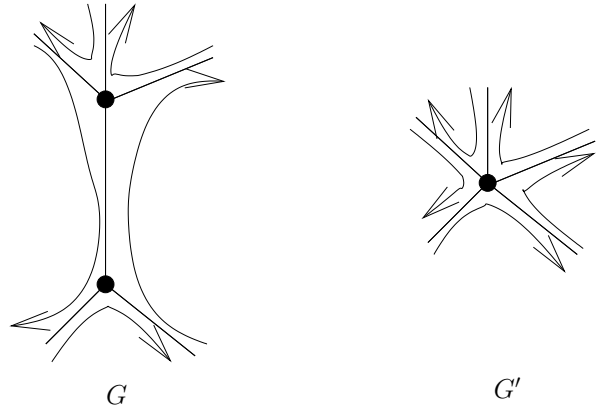}
}}
\caption{The effect of replacing $G$ by $G'$, as in Remark~\ref{rem:Contract}} \label{Contraction}
\end{figure}

A plabic graph is called \newword{reduced} \cite[Section 13]{Post} if the following holds:
\begin{enumerate}
\item The strands cannot be closed loops in the interior of the graph.
\item No strand passes through itself.  The only exception is that we allow simple loops that 
start and end at a boundary vertex $i$.  
\label{NoLoop}
\item For any two strands $\alpha$ and $\beta$, if $\alpha$ and $\beta$ have two common vertices $A$ and $B$,
then one strand, say $\alpha$, is directed from $A$ to $B$, and the other strand $\beta$ is directed from $B$ to $A$.
(That is the crossings of $\alpha$ and $\beta$ occur in opposite orders in the two strands.)
\label{OppositeOrder}
\end{enumerate}

\begin{remark}
Strands, and the reducedness condition, occurs in the physics-inspired literature on quivers. In this literature, strands are called zig-zag paths. Graphs which we call ``reduced'' are said to ``obey condition $Z$" in~\cite[Section 8]{Bock} and are called ``marginally geometrically consistent'' in~\cite[Section 3.4]{Broom}.
\end{remark}

The strand which \emph{ends} at the boundary vertex $i$ is called strand $i$.

\begin{definition} \cite[Section~13]{Post}   For a reduced plabic graph $G$,
let $\pi_G \in S_n$ be the permutation such that
the strand that starts at the boundary vertex $i$ ends at the boundary vertex
$\pi_G(i)$.  A fixed point $\pi_G(i) = i$ corresponds to simple loop at the boundary vertex $i$.
We color a fixed point $i$ of $\pi_G$ as follows: $col(i)=1$ if the corresponding loop is counter-clockwise;
and $col(i) = -1$ if the loop is clockwise.
In this way, we assign the \newword{decorated strand permutation} $\pi_G^{:} = (\pi_G,col)$ to each
reduced plabic graph $G$.
\end{definition}

We will label the faces of a reduced plabic graph with subsets of $[n]$; this construction was first published in~\cite{Scott}.  By condition~\ref{NoLoop}  each strand divides the disk into two parts. For each face $F$ we label that $F$ with the set of those $i \in [n]$ such that $F$ lies to the left of strand $i$. See Figure~\ref{fig:pla_label} for an example. So given a plabic graph $G$, we define $\f(G)$ as the set of labels that occur on each face of that graph.

When we pass from one face $F$ of $G$ to a neighboring one $F'$, we cross two strands. For one of these strands, $F$ lies on its left and $F'$ on the right; for the other $F$ lie on the right and $F'$ on the left. So every face is labeled by the same number of strands as every other.
We define this number to be the \newword{rank} of
the graph; it will eventually play the role of $k$.

%Our description of maximal weakly separated collections in terms of graphs is the following:
The following claim establishes a correspondence between maximal weakly separated collections in 
a positroid and reduced plabic graphs.  It describes maximal weakly separated collections as labeled
sets $\f(G)$ of reduced plabic graphs.

\begin{theorem} \label{thm:WSisPG}
%A collection $\C$ is a maximal weakly separated collection if and only if it is
%$\f(G)$ for a reduced plabic graph with strand permutation $[k+1, k+2, \ldots, n,
%1,2,\ldots, k]$. 
For a decorated permutation $\pi^:$ and the corresponding Grassmann necklace $\I= \I(\pi^:)$,
a collection  $\C$ is a maximal weakly separated collection
inside the positroid $\M_{\I}$ if and only if it has the form $\C=\f(G)$ for a reduced plabic graph
with strand permutation $\pi^:$.  

In particular, a maximal weakly separated collection $\C$ in $\binom{[n]}{k}$ has the form $\C = \f(G)$
for a reduced plabic graph $G$ with strand permutation $w^: = [k+1, k+2, \ldots, n, 1,2,\ldots, k]$. 
\end{theorem}

%
%This labeling is very important, because we will later show that each maximal separated collection of $\binom{[n]}{k }$ can be given by $\f_{G}$ for a plabic graph corresponding to permutation $[k+1,\cdots,n,1,\cdots,k]$.

We define the boundary face at $i$ to be the face touching the part of the disk between boundary vertices $i-1$ and $i$ in clockwise order. Let $I_i$ be the label of that face. At boundary point $i$, the strand $i$ comes in and the strand $\pi(i)$ leaves. So $I_{i+1}$ is obtained from $I_i$ by deleting $i$ and adding in $\pi(i)$; we deduce that $(I_1,\cdots,I_n)$ is the Grassmann necklace $\I(\pi)$. 

We now describe how to see mutations in the context of plabic graphs. We have following 3 moves on plabic graphs. 

\medskip

(M1) Pick a square with vertices alternating in colors, such that all vertices have degree $3$. We can switch the colors of all the vertices. See Figure~\ref{fig:square_move}. 
\begin{figure}[htbp]
\centerline{\psfrag{l}[cc][cc]{$\longleftrightarrow$}
\includegraphics{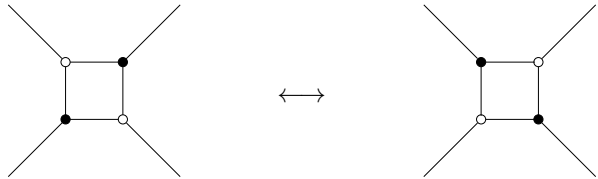}}
\caption{(M1) Square move}
\label{fig:square_move} \end{figure}

(M2) For two adjoint vertices of the same color, we can contract them into one vertex. See Figure~\ref{fig:edge_contraction}. 
\begin{figure}[htbp]
\centerline{\psfrag{l}[cc][cc]{$\longleftrightarrow$}
\includegraphics{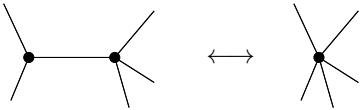}}
\caption{(M2) Unicolored edge contraction}
\label{fig:edge_contraction} \end{figure}

(M3) We can insert or remove a vertex inside any edge. See Figure~\ref{fig:vertex_removal}.
\begin{figure}[htbp]
\centerline{\psfrag{l}[cc][cc]{$\longleftrightarrow$}
\includegraphics{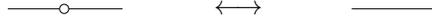}}
\caption{(M3) Vertex removal}
\label{fig:vertex_removal} \end{figure}

The moves do not change the associated decorated permutation of the plabic graph, and do not change whether or not the graph is reduced. The power of these moves is reflected in the next Theorem:

\begin{theorem} \cite[Theorem 13.4]{Post}
\label{thm:P}
 Let $G$ and $G'$ be two reduced plabic graphs with the same number of boundary vertices. Then the following claims are equivalent:
\begin{itemize}
\item $G$ can be obtained from $G'$ by moves (M1)-(M3).
\item These two graphs have the same decorated strand permutation $\pi_G^{:} = \pi_{G'}^{:}$.
\end{itemize}
\end{theorem}

Moves (M2) and (M3) do not change $\f(G)$, while move (M1) changes $\f(G)$ by a mutation. So, once we prove Theorem~\ref{thm:WSisPG}, we will establish that all maximal weakly separated collections within $\M(\I)$ are connected to each other by mutations. 

Notice also that the moves (M1)-(M3) do not change the number of faces of the plabic graph. So all reduced plabic graphs for a given decorated permutation have the same number of faces. 
This number was calculated in~\cite{Post}:

\begin{theorem} 
\label{thm:P2}
Let $G$ be a reduced plabic graph with decorated permutation $\pi^{:}$. Then $G$ has $\ell(\pi^{:}) + 1$ faces.
\end{theorem}

\begin{proof}
By~\cite[Proposition 17.10]{Post}, $\ell(\pi)$ is the dimension of $S_{\M}^{tnn}$, a manifold whose definition we do not need to know. By~\cite[Theorem 12.7]{Post}, $S_{\M}^{tnn}$ is isomorphic to $\RR^{|F(G)| - 1}$. 
\end{proof}

Thus, proving Theorem~\ref{thm:WSisPG} will establish all parts of~\ref{thm:GSConjecture}, thus proving the conjectures of Scott and of Leclerc and Zelevinsky.

\section{Consequences for cluster algebras} \label{sec:Clust}

This section is not cited in the rest of the paper.

From the beginnings of the theory of cluster algebras, it has been expected that the coordinate ring of the Grassmannian, in its Pl\"ucker embedding, would be a cluster algebra, and that the Pl\"ucker coordinates would be cluster variables. 
This was verified by Scott in~\cite{Scott2}.
Moreover, Scott showed that, given any reduced alternating strand diagram, the face labels of that diagram form a cluster in the cluster structure for $G(k,n)$.
With the tools of this paper, we can establish the converse statement.

\begin{theorem} \label{Cluster}
Let $\C$ be a subset of $\binom{[n]}{k}$. The following are equivalent:
\begin{enumerate}
\item The set of Pl\"ucker coordinates $\{ p_I \}_{I \in \C}$ is a cluster in the cluster algebra structure on the coordinate ring of $G(k,n)$.
\item $\C$ is a maximal weakly separated collection.
\item $\C$ is the collection of face labels of a reduced plabic graph for the uniform matroid $\binom{[n]}{k}$.
\end{enumerate}
\end{theorem}

In this section, we use the language of cluster algebras freely.

The implication $(2) \implies (3)$ is the main result of this paper; the implication $(3) \implies (1)$ is~\cite[Theorem 2]{Scott2}. We now show  $(1) \implies (3)$.

Let $I$ and $J \in \C$; we will show that $I \parallel J$.
Let $\mathbb{C}_q(G(k,n))$ be the quantization of the coordinate ring of $G(k,n)$ introduced in~\cite{QA}; since $p_I$ and $p_J$ are in a common cluster, they quasi-commute in this ring.
By~\cite{GLS}, the ring $\mathbb{C}_q(G(k,n))$ is isomorphic to the ring of ``quantum minors" of a $k \times (n-k)$ matrix (see~\cite{Scott} or~\cite{LZ}), so the quantum minors $\Delta(I)$ and $\Delta(J)$ quasi-commute.
As computed in~\cite{Scott}, if $\Delta(I)$ and $\Delta(J)$ quasi-commute, then $I$ and $J$ are weakly separated.

So we know that $\C$ is a weakly separated collection.
Every cluster has cardinality $1+k(n-k)$, so $\C$ is a maximal weakly separated collection by the bound of~\cite{Scott}. \qedsymbol

We expect that an analogous statement holds for all positroids.
We have not proved it here for two reasons.
The first is that, although positroid varieties are widely expected to have a cluster structure, this has not yet been verified in print.
The second is that, since the coordinate ring of a positroid variety is a quotient of the coordinate ring of $G(k,n)$, it is possible that two minors which do not quasi-commute on $G(k,n)$ do quasi-commute when restricted to this smaller subvariety. 
Deriving the analogues of~\cite{Scott} and~\cite{LZ} for positroid varieties strikes us as an excellent project.

\section{Plabic graphs and weakly separated sets}
\label{sec:PGraphs}

Our goal in this section is to show that, for any reduced plabic graph $G$ with associated decorated permutation $\pi^:$, the collection $\f(G)$ is a weakly separated collection for $\M_{\I(\pi^:)}$.
This is the easy part of Theorem~\ref{thm:WSisPG}. We will show that this collection is maximal in section~\ref{sec:tilings}, and that any maximal weakly separated collection is of the form $\f(G)$ in section~\ref{sec:final}.

\begin{proposition}\label{prop:plagws}
Let $G$ be a reduced plabic graph. Then $\f(G)$ is a weakly separated collection.
\end{proposition}

\begin{proof}
Assume $\f(G)$ is not weakly separated. Pick $I$ and $J \in \f(G)$ such that $I \not \parallel J$. Let $a$, $b$, $c$ and $d$ be cyclically ordered elements such that $\{a,c\} \subset I \setminus J$ and $\{b,d\} \subset J \setminus I$. 

We first consider the case where strands $a$ and $c$ don't cross (see Figure~\ref{DontCross}).
In this case, region $I$ is to the left of strands $a$ and $c$, and region $J$ is to their right, so strands $a$ and $c$ must be parallel, not antiparallel. 
In other words, the endpoints $( \pi^{-1}(a), a, b, c, \pi^{-1}(c))$ are circularly ordered.
If $\pi^{-1}(b)$ is in $(b, \pi^{-1}(a))$, then strands $a$ and $b$ cannot cross and $I$ is on the left side of $b$, contrary to our desires.
If $\pi^{-1}(b)$ is in $(\pi^{-1}(c), b)$ then strand $b$ cannot cross strand $c$ and we deduce that $J$ is to the right of strand $b$, again contrary to our desires.  
But the intervals $(b, \pi^{-1}(a))$ and $(\pi^{-1}(a), c)$ cover all of $[n]$, so this excludes all possible positions for $\pi^{-1}(b)$ and we have a contradiction.
This concludes the proof in the case that strands $a$ and $c$ do not cross.

\begin{figure}
\centerline{
\psfrag{I}[cc][cc]{\Huge $I$}
\psfrag{J}[cc][cc]{\Huge $J$}
\psfrag{a}[cc][cc]{\Huge $a$}
\psfrag{c}[cc][cc]{\Huge $c$}
\scalebox{0.5}{\includegraphics{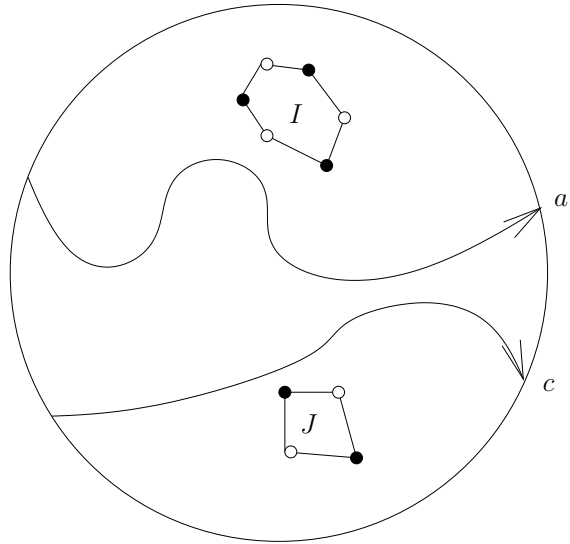}}}
\caption{The case where strands $a$ and $c$ don't cross} \label{DontCross}
\end{figure}

Now, suppose that strands $a$ and $c$ cross (see Figure~\ref{Lemma71}).
 Let $R$ be the region which is to the left of $a$ and right of $c$; let $T$ be to the right of $a$ and left of $c$, and let $S_1$, $S_2$, \dots, $S_s$ be the regions which are both to the left of $a$ and $c$ or both to the right. Our numbering is such that strand $a$ first passes by $S_1$ and precedes in increasing order, while $c$ starts at $S_s$ and goes in decreasing order. Let face $I$ be in $S_i$ and $J$ be in $S_j$; note that $i \neq j$. We will discuss the case that $i < j$; the case that $i>j$ is equivalent by relabeling $(a,b,c,d)$ as $(c,d,a,b)$.

\begin{figure}
\centerline{
\psfrag{I}[cc][cc]{\Huge $I$}
\psfrag{J}[cc][cc]{\Huge $J$}
\psfrag{R}[cc][cc]{\Huge $R$}
\psfrag{S1}[cc][cc]{\Huge $S_1$}
\psfrag{S2}[cc][cc]{\Huge $S_2$}
\psfrag{S3}[cc][cc]{\Huge $S_3$}
\psfrag{S4}[cc][cc]{\Huge $S_4$}
\psfrag{T}[cc][cc]{\Huge $T$}
\psfrag{a}[cc][cc]{\Huge $a$}
\psfrag{b}[cc][cc]{\Huge $b$}
\psfrag{c}[cc][cc]{\Huge $c$}
\psfrag{d}[cc][cc]{\Huge $d$}
\psfrag{Q}[cc][cc]{\Huge \mbox{strand $d$}}
\scalebox{0.5}{\includegraphics{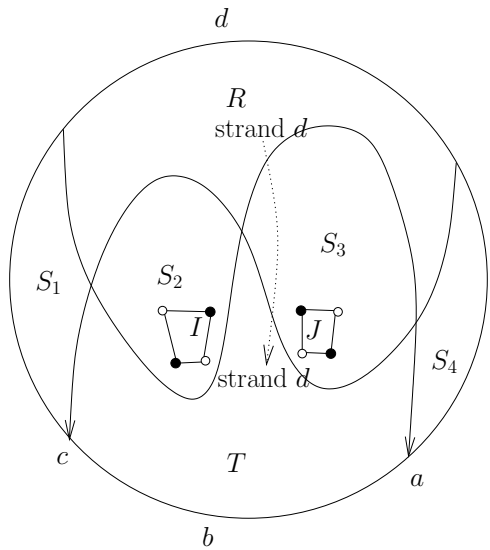}}}
\caption{The case where strands $a$ and $c$ do cross} \label{Lemma71}
\end{figure}

Consider strand $d$, as it passes through the various regions $R$, $S_1$, \dots, $S_s$ and $T$. (In Figure~\ref{Lemma71}, a portion of strand $d$ is shown as a dotted line.) We claim that it is impossible that strand $d$ both passes from $R$ to $S_e$ to $T$ and from $T$ to $S_f$ to $R$. The proof is simple: If it did, either its crossings with strand $a$ or its crossings with $c$ would violate condition~\ref{OppositeOrder} in the definition of reducedness. On the other hand, strand $d$ must enter both $R$ and $T$, as otherwise regions $I$ and $J$ would lie on the same side of strand $d$. Let $m$ be the index such that strand $d$ travels from $R$ to $S_m$ to $T$, or else from $T$ to $S_m$ to $R$. We must have $i \leq m \leq j$, since $J$ and $I$ are opposite sides of $d$, and $d$ must travel from $R$ to $S_m$ to $T$, as $I$ is on the right of $d$ and $J$ on the left.  

But the boundary point $d$, where strand $d$ terminates, is in the cyclic interval $(c,a)$. So $d$ cannot end in $T$. It also cannot end in $R$, as it goes from $R$ to $S_m$ to $T$. 
So it must end in $S_1$ or $S_s$. We discuss the former case, since the latter case is similar. 
If $d$ ends in $S_1$, it must come from $T$, passing through $a$. But then the intersections of $c$ and $d$ violate condition~\ref{OppositeOrder}.
\end{proof}

We now use Proposition~\ref{prop:plagws}, together with Theorem~\ref{thm:P}, to demonstrate some facts about the face labels of plabic graphs. %The reader should consult Figure~\ref{M1} while reading propositions~\ref{prop:DistinctLabels} and~\ref{prop:LabelsInM}. 
%Note that the strands in this figure are labeled where they exit the edges of the figure.

\begin{figure}
\centerline{\scalebox{0.5}{
\psfrag{Sab}{\Huge $Sab$}
\psfrag{Sbc}{\Huge $Sbc$}
\psfrag{Scd}{\Huge $Scd$}
\psfrag{Sad}{\Huge $Sad$}
\psfrag{Sac}{\Huge $Sac$}
\psfrag{Sbd}{\LARGE $Sbd$}
\psfrag{a}{\Huge $a$}
\psfrag{b}{\Huge $b$}
\psfrag{c}{\Huge $c$}
\psfrag{d}{\Huge $d$}
\includegraphics{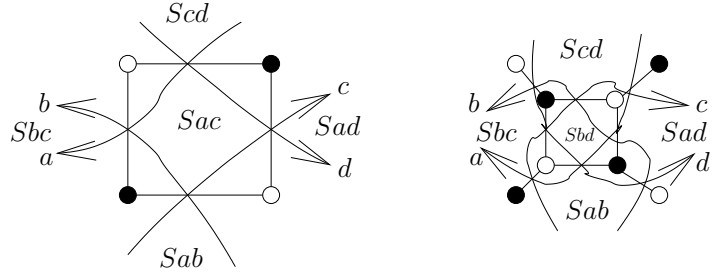}}}
\caption{The effect of (M1) on face labels} \label{M1}
\end{figure}

Now we want to show that $\F(G)$ is actually a weakly separated collection of $\M_{\I(\pi^:)}$. We need the following proposition.

\begin{prop} \label{prop:Reduction}
Let $G$ be a reduced plabic graph with decorated permutation $\pi^:$ such that $\ell(\pi^:)>0$. There is some pair of strands $(i,j)$ in $\pi^{:}$ which only crosses once, at an edge $e$ of $G$, and such that, if we delete the edge $e$, thus uncrossing the strands, the resulting graph $G'$ is still reduced.
\end{prop}

\begin{proof}
According to~\cite[Lemma 18.9]{Post}, a \emph{removable} edge has these properties. (See~\cite{Post} for definitions.) According to~\cite[Corollary 18.10]{Post}, the removable edges of $G$ are in bijection with the decorated permutations covered by $\pi^{\colon}$. Since $\ell(\pi^{\colon})>1$, it covers at least one decorated permutation.
\end{proof}

\begin{prop} 
\label{prop:PartConverse}

If $G$ is a reduced plabic graph with decorated permutation $\pi^:$, then the following properties hold:
\begin{enumerate}

\item The boundary cells of $G$ are labeled by $\I(\pi^:)$.
\item Every face of $G$ receives a separate label in $\F(G)$.
\item $\F(G)$ is contained in $\M_{\I(\pi^:)}$.
\end{enumerate}
\end{prop}

\begin{proof}
We are going to use induction on $\ell(\pi^:)$. When $\ell(\pi^:)=0$, any reduced plabic graph has only one face due to Theorem~\ref{thm:P2}. This is possible only when $\pi$ is the identity. Then there is only one face, and it has the label $I_1 = I_2 = \cdots = I_n$. This covers the base case for our induction argument.

Let $i$, $j$, $e$ and $G'$ be as in Proposition~\ref{prop:Reduction}.
By induction hypothesis, all the properties hold for $G'$. We will use $\mu^{:}$ to denote the decorated permutation of $G'$ and use $\I'=(I_1',\cdots,I_n')$ for the Grassmann necklace of $G'$. %The strands $i$ and $j$ divide $G$ into $4$ regions, and we denote them by $R_{i,j},R_{j,\pi^{-1}(i)},R_{\pi^{-1}(i),\pi^{-1}(j)},R_{\pi^{-1}(j),i}$.

We start with the first property. Swapping $i$ and $j$ going from $\mu^:$ to $\pi^:$ corresponds to changing $\I'$ to $\I$ in a way that, if some ${I'}_k$ contains only one of $i$ and $j$, than it gets swapped with the other. When going from $G'$ to $G$, the label of a boundary face gets changed only when it contains exactly one of $i$ and $j$, and the change is by swapping one with the other. So the boundary cells of $G$ are labeled by $\I(\pi^:)$.

Now we check the second property. Since $i$ and $j$ only cross once, if two faces are assigned the same label in $G$, that means they have to be in the same region with respect to strands $i$ and $j$. But this also means they are assigned the same label in $G'$, a contradiction.

We prove the third property directly. Assume for the sake of contradiction that we have some $J \in \F(G)$ such that $J \not \geq_i I_i$. For $J = \{j_1,\cdots,j_k\}$ and $I_i = \{h_1,\cdots,h_k\}$, let $t$ denote the first position such that $j_t <_i h_t$. From this, we get $j_t \not \in I_i$, and hence $\pi^{-1}(j_t) \leq_i j_t$. And for any $q \in [t,k]$, we have $\pi^{-1}(h_t) \geq_i h_t$, and it follows from comparing the strands $j_t$ and $h_q$ that $J$ should contain $h_q$. But this implies that $|J| = k+1$, a contradiction.

%We will now check the second property. Assume for the sake of contradiction that some $J$ and $H$ in $\F(G)$ are not weakly separated. Let $J'$ and $H'$ denote the labels of the corresponding faces in $G'$. Then we should not have the case when $J \setminus H = J' \setminus H'$ and $H \setminus J = H' \setminus J'$. So the two faces should be in different regions among $R_{i,j},R_{j,\pi^{-1}(i)},R_{\pi^{-1}(i),\pi^{-1}(j)},R_{\pi^{-1}(j),i}$, since $\F(G')$ is a weakly separated collection. 

\end{proof}

Combining the first and third property of Proposition~\ref{prop:PartConverse} with Proposition~\ref{prop:plagws}, we achieve the following claim:

\begin{corollary} \label{cor:PartConverse}
If $G$ is a reduced plabic graph with boundary $\I$, then $\f(G)$ is a weakly spearated collection in $\M_{\I}$.
\end{corollary}

\begin{proof}
We have $\I \subseteq \f(G)$ by Proposition~\ref{prop:PartConverse}; that $\I$ is weakly separated by Proposition~\ref{prop:plagws}; and that $\f(G) \subseteq \M_{\I}$ by Proposition~\ref{prop:PartConverse}.
\end{proof}

We have not yet shown that $\f(G)$ is maximal; that will be Theorem~\ref{thrm:IsMaximal}.

\section{Plabic tilings} \label{sec:tilings}
Given a maximal weakly separated collection $\C$ of $\binom{[n]}{k}$, we need to construct a plabic graph $G$ such that $\f(G) = \C$. To do so, we will define a \newword{plabic tiling}. For a weakly separated collection $\C$, we will construct a $2$-dimensional CW-complex embedded in $\RR^2$, and denote this complex by $\Sigma(\C)$.  Maximal plabic tilings will turn out to be dual to reduced bipartite plabic graphs. We will associate a plabic tiling to any weakly separated collection, maximal or not.
See Example~\ref{IncompleteExample} for a discussion of $\Sigma(\C)$ for a non-maximal $\C$.

Figure~\ref{TilingExample} shows $\Sigma(\C)$ for $\C$ the collection of face labels in Figure~\ref{fig:pla_label}. 
The points $v_i$, which we will introduce soon, are shown in the lower right of Figure~\ref{TilingExample}. 
They are shown in the same scale as the rest of the figure.

\begin{figure}
\centerline{\scalebox{0.45}{
\psfrag{123}[cc][cc]{\Huge $123$}
\psfrag{234}[cc][cc]{\Huge $234$}
\psfrag{345}[cc][cc]{\Huge $345$}
\psfrag{456}[cc][cc]{\Huge $456$}
\psfrag{567}[cc][cc]{\Huge $567$}
\psfrag{678}[cc][cc]{\Huge $678$}
\psfrag{178}[cc][cc]{\Huge $178$}
\psfrag{128}[cc][cc]{\Huge $128$}
\psfrag{127}[cc][cc]{\Huge $127$}
\psfrag{137}[cc][cc]{\Huge $137$}
\psfrag{136}[cc][cc]{\Huge $136$}
\psfrag{135}[cc][cc]{\Huge $135$}
\psfrag{134}[cc][cc]{\Huge $134$}
\psfrag{167}[cc][cc]{\Huge $167$}
\psfrag{156}[cc][cc]{\Huge $156$}
\psfrag{145}[cc][cc]{\Huge $145$}
\includegraphics{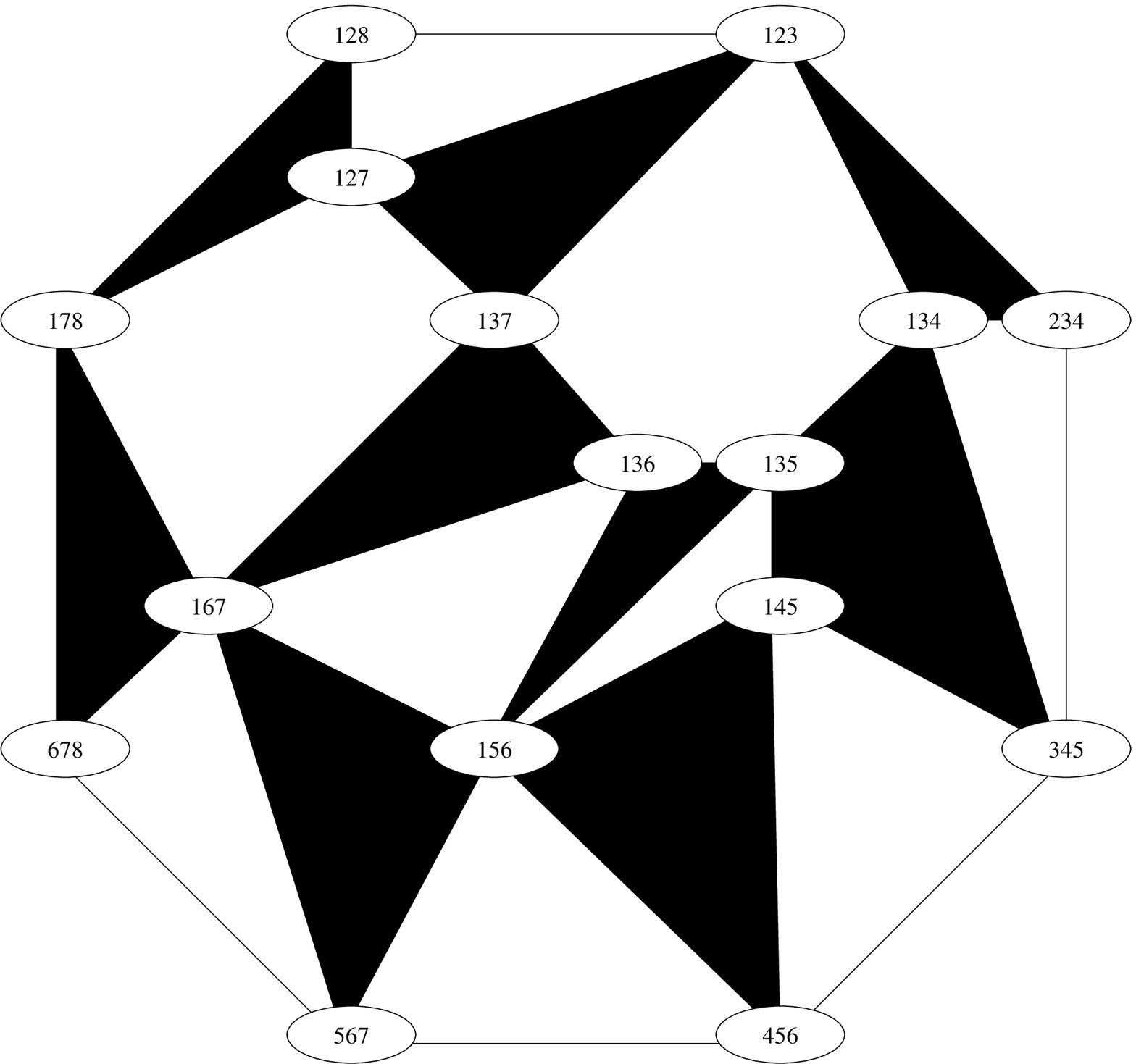}}
\quad
\scalebox{0.45}{
\psfrag{v1}[cc][cc]{\Huge $v_1$}
\psfrag{v2}[cc][cc]{\Huge $v_2$}
\psfrag{v3}[cc][cc]{\Huge $v_3$}
\psfrag{v4}[cc][cc]{\Huge $v_4$}
\psfrag{v5}[cc][cc]{\Huge $v_5$}
\psfrag{v6}[cc][cc]{\Huge $v_6$}
\psfrag{v7}[cc][cc]{\Huge $v_7$}
\psfrag{v8}[cc][cc]{\Huge $v_8$}
\includegraphics{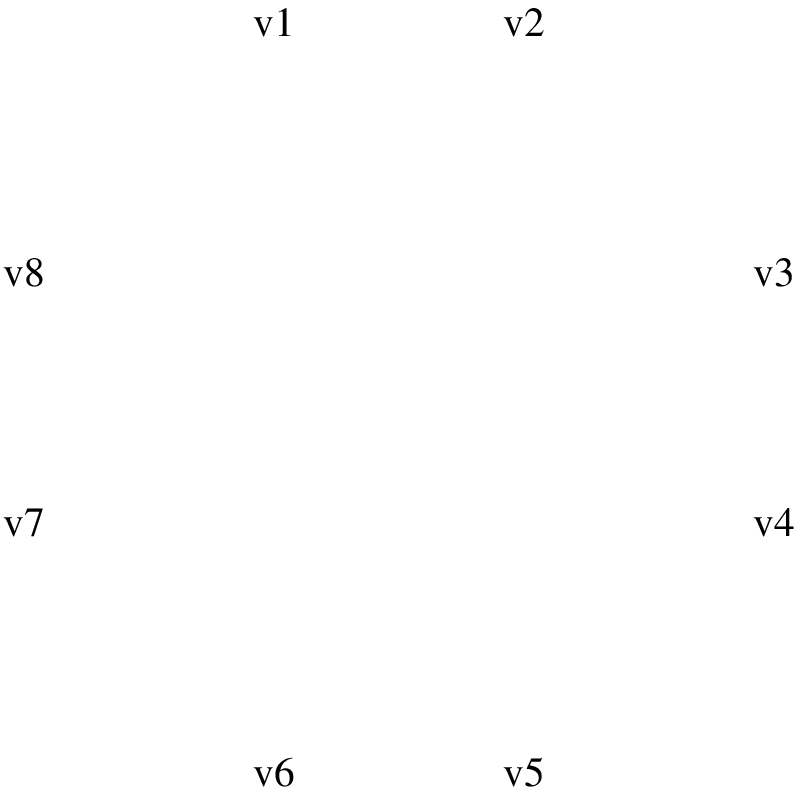}}}
\caption{$\Sigma(\C)$ for $\C$ the face labels of figure~\ref{fig:pla_label}, and the points $v_i$}\label{TilingExample}
\end{figure}

Let us fix $\C$, a weakly separated collection in $\binom{[n]}{k}$. For $I$ and $J \in \binom{[n]}{k}$, say that $I$ \newword{neighbors} $J$ if 
$$|I \setminus J|=|J \setminus I|=1.$$
Let $K$ be any $(k-1)$ element subset of $[n]$. We define the \newword{white clique} $\W(K)$ to be the set of $I \in \C$ such that $K \subset I$. Similarly, for $L$ a $(k+1)$ element subset of $[n]$, we define the \newword{black clique} $\B(L)$ for the set of $I \in \C$ which are contained in $L$. We call a clique \newword{nontrivial} if it has at least three elements. 
Observe that, if $\mathcal{X}$ is a nontrivial clique, then it cannot be both black and white. 

\begin{remark}
Let $G$ be a reduced plabic graph of rank $k$. The black (respectively white) cliques of $\f(G)$ correspond to the black (white) vertices of $G$. More precisely, for each vertex $v$, the labels of the faces bordering $v$ form a clique, and all cliques are of this form. %See Proposition~\ref{prop:duality}
\end{remark}

Observe that a white clique $\W(K)$ is of the form $\{ K a_1, K a_2, \ldots, K a_r \}$ for some $a_1$, $a_2$, \dots, $a_r$, which we take to be cyclically ordered. Similarly, $\B(L)$ is of the form $\{ L \setminus b_1, L \setminus b_2, \ldots, L \setminus b_s \}$, with the $b_i$'s cyclically ordered. 
If $\W(K)$ is nontrivial, we define the boundary of $\W(K)$ to be the cyclic graph 
$$(K a_1) \to (K a_2) \to \cdots \to (K a_r) \to (K a_1).$$ 
Similarly, the boundary of a nontrivial $\B(L)$ is 
$$(L \setminus b_1) \to (L \setminus b_2) \to \cdots \to (L \setminus b_s) \to (L \setminus b_1).$$ 
If $(J, J')$ is a two element clique, then we define its boundary to be the graph with a single edge $(J,J')$; we define an one element clique to have empty boundary.

\begin{lemma} \label{lem:nbhr}
Let $I$ neighbor $J$; set $K = I \cap J$ and $L = I \cup J$. 
If $\W(K)$ and $\B(L)$ are nontrivial, then there is an edge between $I$ and $J$ in the boundaries of $\W(K)$ and $\B(L)$.
\end{lemma}

\begin{proof}
Let $i$ be the lone element of $I \setminus J$ and $j$ the lone element of $J \setminus I$. 
Let $\W(K) = \{ K a_1, K a_2, \ldots, K a_r \}$ and $\B(L) = \{ L \setminus b_1, L \setminus b_2, \ldots, L \setminus b_s \}$. 
So $i$ and $j$ occur among the $a$'s, and among the $b$'s. Consider the four sets:

\begin{eqnarray*} 
S_1 := \{ a_m : a_m \in (i, j) \} & \quad & S_2 := \{ a_m :  a_m \in (j, i) \} \\
S_3 := \{ b_m : b_m \in (i, j) \} & \quad & S_4 := \{ b_m :  b_m \in (j, i) \} 
\end{eqnarray*}

Our goal is to show that either $S_1 = S_3 = \emptyset$ or $S_2 = S_4 = \emptyset$. 

Suppose (for the sake of contradiction) that $S_1$ and $S_4$ are both nonempty, with $a \in S_1$ and $b \in S_4$.
Set $P = K a$ and $Q = L \setminus b$. Then $a \in P \setminus Q$, $i \in Q \setminus P$, $b \in P \setminus Q$ and $j \in Q \setminus P$, so $P$ and $Q$ are not weakly separated.
We have a contradiction and we deduce that at least one of $S_1$ and $S_4$ is empty.
Similarly, at least one of $S_2$ and $S_3$ is empty.
On the other hand, $\W(K)$ is nontrivial, so at least one of $S_1$ and $S_2$ is nonempty. 
Similarly, at least one of $S_3$ and $S_4$ is nonempty. 

%Combining the above statements, we deduce that either $S_1$ and $S_3$ are empty, and $S_2$ and $S_4$ are not, or vice versa.
Since at least one of $(S_1, S_4)$ is empty, and at least one of $(S_2, S_3)$ is, we deduce that at least two of $(S_1, S_2, S_3, S_4)$ are empty.
Similarly, at most two of them are empty, so precisely two of the $S_i$'s are empty.
Checking the $6$ possibilities, the ones consistent with the above restrictions are that either $S_1$ and $S_3$ are empty, and $S_2$ and $S_4$ are not, or vice versa.

\end{proof}

We now define a two dimensional CW-complex $\Sigma(\C)$.
The vertices of $\Sigma(\C)$ will be the elements of $\C$.
There will be an edge $(I,J)$ if
\begin{enumerate}
\item $\W(I \cap J)$ is nontrivial and $(I,J)$ appears in the boundary of $\W(I \cap J)$ or
\item $\B(I \cup J)$ is nontrivial and $(I,J)$ appears in the boundary of $\B(I \cup J)$ or
\item $\W(I \cap J) = \B(I \cup J)  = \{ I, J \}$.\footnote{This third case is only important in certain boundary cases involving disconnected positroids. For most of the results of this paper, it is not important whether or not we include an edge in this case.}
\end{enumerate}
There will be a (two-dimensional) face of $\Sigma(\C)$ for each nontrivial clique $\mathcal{X}$ of $\C$. 
The boundary of this face will be the boundary of $\mathcal{X}$. 
By Lemma~\ref{lem:nbhr}, all edges in the boundary of $\mathcal{X}$ are in $\Sigma(\C)$, so this makes sense.
We will refer to each face of $\Sigma(\C)$ as \newword{black} or \newword{white}, according to the color of the corresponding clique. We call a CW-complex of the form $\Sigma(\C)$ a \newword{plabic tiling}. So far, $\Sigma(\C)$ is an abstract CW-complex. Our next goal is to embed it in a plane.

Fix $n$ points $v_1$, $v_2$, \dots, $v_n$ in $\RR^2$, at the vertices of a convex $n$-gon in clockwise order. 
Define a linear map $\pi: \RR^n \to \RR^2$ by $e_{a} \mapsto v_{a}$. 
For $I \in \binom{[n]}{t}$, set $e_I = \sum_{a \in I} e_{a}$. 

We will need a notion of weak separation for vectors in $\RR^{[n]}$. Let $e$ and $f$ be two vectors in $\RR^{[n]}$, with $\sum_{a \in [n]} e_{a} = \sum_{a \in [n]} f_{a}$. 
 We define $e$ and $f$ to be \newword{weakly separated} if there do \emph{not} exist  $a$, $b$, $a'$ and $b'$, cyclically ordered, with $e_{a} > f_{a}$, $e_{b} < f_{b}$, $e_{a'} > f_{a'}$ and $e_{b'} < f_{b'}$. 
 So, for $I$ and $J \in \binom{[n]}{k}$, we have $I \parallel J$ if and only if $e_I \parallel e_J$.

\begin{lemma}\label{lem:vecinj}
Let $e$ and $f$ be two different points in $\RR^n$, with  $\sum_{a \in [n]} e_{a} = \sum_{a \in [n]} f_{a}$ and $e \parallel f$. Then $\pi(e) \neq \pi(f)$.
\end{lemma}

\begin{proof}
Since $e$ and $f$ are weakly separated, there are some $\ell$ and $r$ in $[n]$ such that $e_{a} - f_{a}$ is nonnegative for $a$ in $[\ell, r)$ and is nonpositive for $a$ in $[r, \ell)$. We have 
\begin{equation}
\pi(e) - \pi(f) = \sum_{a \in [\ell, r)} (e_{a} - f_{a}) v_{a} - \sum_{b \in [r, \ell)} (f_{a} - e_{a}) v_{a}. \label{dipslacement}
\end{equation}
Since  $\sum_{a \in [n]} e_{a} = \sum_{a \in [n]} f_{a}$, the right hand side of the above equation is a positive linear combination of vectors of the form $v_{a} - v_{b}$, with $a \in [\ell, r)$ and $b \in [r, \ell)$. 

Since the $v_a$'s are the vertex of a convex $n$-gon, there is a line $\lambda$ in $\RR^2$ separating $\{ v_{a} : a \in [\ell, r) \}$ from $\{ v_{b} : b \in [r, \ell) \}$. So every vector $v_{a} - v_{b}$ as above crosses from the former side of $\lambda$ to the latter. A positive linear combination of such vectors must cross the line $\lambda$, and can therefore not be zero. So the right hand side of the above equation is nonzero, and we deduce that $\pi(e) \neq \pi(f)$.
\end{proof}

We extend the map $\pi$ to a map from $\Sigma(\C)$ to $\RR^2$ as follows: Each vertex $I$ of $\Sigma(\C)$ is sent to $\pi(e_I)$ and each face of $\Sigma(\C)$ is sent to the convex hull of the images of its vertices. 
We encourage the reader to consult Figure~\ref {TilingExample} and see that the vector $\pi(Si) - \pi(Sj)$ is a translation of $v_i - v_j$.

\begin{proposition}\label{prop:embed}
For any weakly separated collection $\C$, the map $\pi$ embeds the CW-complex $\Sigma(\C)$ into $\RR^2$.
\end{proposition}

\begin{proof}
Suppose, to the contrary, that there are two faces of $\Sigma(\C)$ whose interiors have overlapping image. We will deal with the case that both of these faces are two dimensional, with one white and the other black. The other possibilities are similar, and easier. Let the vertices of the two faces be $\{ K a_1, K a_2, \ldots, K a_r \}$ and $\{ L \setminus b_1, L \setminus b_2, \ldots, L  \setminus b_s \}$. Write $A = \{ a_1, \ldots, a_r \}$ and $B = \{ b_1, \ldots, b_s \}$. 

Before we analyze the geometry of $\pi$, we will need to do some combinatorics. Note that, if $K \subset L$, then the faces are convex polygons with a common edge, lying on opposite sides of that edge, and thus have disjoint interiors.\footnote{The fact that we have to consider this possibility is one of the things that makes the case of two faces of opposite color more difficult.} So we may assume that $K \setminus L$ is nonempty. Also, $|L|=k+1 > |K|=k-1$, so $L \setminus K$ is nonempty. Let $x \in K \setminus L$ and $y \in L \setminus K$. Suppose that $a_i$ and $b_j$ are in $(x,y)$. Then weak separation of $K a_i$ and $L \setminus b_j$ implies that $a_i \leq_x b_j$. More generally, if $a \in (K \cup A) \setminus L$ and $b \in (L \cup B) \setminus K$ are both in $(x,y)$, we still have $a \leq_x b$. Similarly, if $a \in (K \cup A) \setminus L$ and $b \in (L \cup B) \setminus K$ are both in $(y,x)$, then $a \geq_x b$. So, there is some $c \in [x,y]$ and $d \in [y,x]$ such that $(K \cup A) \setminus L \subseteq [d,c)$ and $(L \cup B) \setminus K \subseteq [c,d)$.

Now, let $\zeta$ be the point which is in the interior of both faces, say
$$\zeta = \sum p_i \pi(K a_i) = \sum q_i \pi(L \setminus b_i)$$
for some positive scalars $p_i$ and $q_i$ with $\sum p_i = \sum q_i =1$.
Define the vectors $u$ and $v$ by $u = \sum p_i e_{K a_i}$ and $v = \sum q_i e_{L \setminus b_i}$, so $\pi(u) = \pi(v)$. 
All the positive entries of $u-v$ are contained in $[d,c)$, and all the negative entries in $[c,d)$.
So $u$ and $v$ are weakly separated and, by Lemma~\ref{lem:vecinj}, $\pi(u) \neq \pi(v)$, a contradiction.
\end{proof}

Looking through the proof, we have proved the more technical result:

\begin{lemma} \label{WriteThis}
Let $P$ and $Q$ be different faces of $\Sigma(\C)$. Let $V(P)$ and $V(Q)$ be the sets of vertices of the faces $P$ and $Q$. Let $u$ be a vector of the form $\sum_{I \in V(P)} c_I e_I$, where $c_I >0$ and $\sum c_I=1$. Similarly, let $v$ be a vector of the form $\sum_{I \in V(Q)} d_I e_I$, with $d_I>0$ and $\sum d_I=1$. Then $u$ and $v$ are weakly separated.
\end{lemma}

\begin{remark}
\label{rem:circ}
Let $K a_1$, $K a_2$, \dots, $K a_r$ be the vertices of a white face of $\C$, with the $a_i$ in cyclic order. Then the vertices $\pi(K a_i)$ appear in clockwise order in the planar embedding. If the vertices of a black face are $L \setminus b_1$, $L \setminus b_2$, \dots, $L \setminus b_r$, with the $b_i$ again in cyclic order, then the vertices $\pi(L \setminus b_i)$ again appear in clockwise order.
This is because negation is an orientation preserving operation on $\RR^2$; see Figure~\ref {TilingExample}.
\end{remark}

%Now let's define a dual graph of a clean reduced plabic graph. Let's fix a plabic graph $G$. Let $G^{*}$ be a graph such that
%\begin{itemize}
%\item there is one vertex for each face of $G$,
%\item each vertex of $G^{*}$ is labeled with the label of the corresponding vertex in $G$,
%\item two vertices are connected by an edge if the corresponding faces are adjancet in $G$ and
%\item a face of $G^{*}$ is colored with the same color of the corresponding vertex in $G$.
%\end{itemize}

%In the next section, we will show that $G^{*}$ is a plabic tiling of $\f(G)$.

Now, we will study the construction $\Sigma(\C)$ when $\C$ is a weakly separated collection for a particular positroid $\M$.
Propositions~\ref{GlueAtPoint} through~\ref{lem:Intervals} are obvious in the case that $\M$ is the largest positroid, $\binom{[n]}{k}$.

Let $\M$ be a positroid and $\I$ the corresponding Grassmann necklace. 
We first give a lemma to reduce to the case when $\I$ is connected: Suppose that $\I$ is not connected, with $I_i=I_j$.
Let $\C$ be a weakly separated collection for $\M^r$; and use the notations $\C^r$, $I^r$ and so forth from Section~\ref{Connected}.
It is easy to check that:
\begin{proposition} \label{GlueAtPoint}
The complex $\Sigma(\C)$ is formed by gluing $\Sigma(\C^1)$ and $\Sigma(\C^2)$ together at the point $I^1 \cup I^2$. 
\end{proposition}

See Figure~\ref{fig:Disconnected} for an example of a disconnected plabic graph and the corresponding plabic tiling.

\begin{figure}
\centerline{\scalebox{0.6}{
\psfrag{02345}[cc][cc]{$02345$}
\psfrag{12345}[cc][cc]{$12345$}
\psfrag{23459}[cc][cc]{$23459$}
\psfrag{24589}[cc][cc]{$24589$}
\psfrag{34589}[cc][cc]{$34589$}
\psfrag{34579}[cc][cc]{$34579$}
\psfrag{45789}[cc][cc]{$45789$}
\psfrag{56789}[cc][cc]{$56789$}
\psfrag{46789}[cc][cc]{$46789$}
\psfrag{0}[cc][cc]{\Huge $0$}
\psfrag{1}[cc][cc]{\Huge $1$}
\psfrag{2}[cc][cc]{\Huge $2$}
\psfrag{3}[cc][cc]{\Huge $3$}
\psfrag{4}[cc][cc]{\Huge $4$}
\psfrag{5}[cc][cc]{\Huge $5$}
\psfrag{6}[cc][cc]{\Huge $6$}
\psfrag{7}[cc][cc]{\Huge $7$}
\psfrag{8}[cc][cc]{\Huge $8$}
\psfrag{9}[cc][cc]{\Huge $9$}
\includegraphics{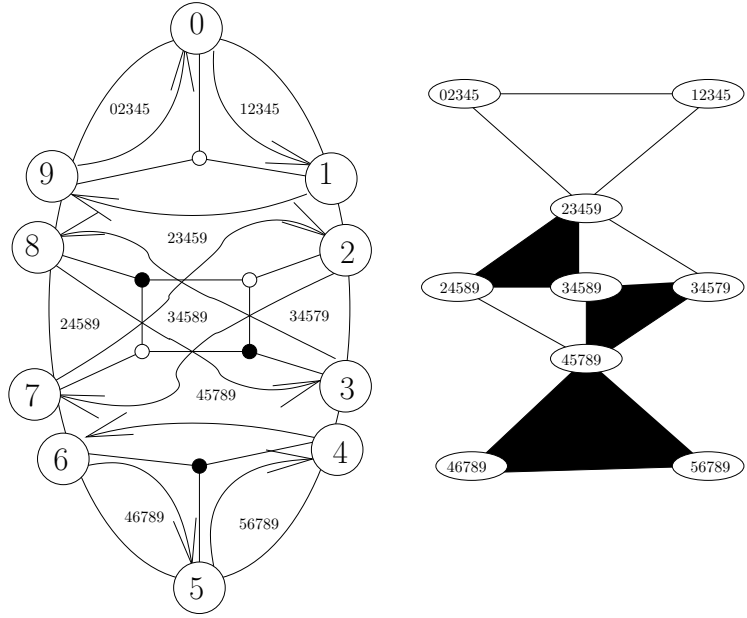}}}
\caption{A disconnected plabic graph and the corresponding plabic tiling} \label{fig:Disconnected}
\end{figure}

We now, therefore, restrict to the case that $\I$ is a connected Grassmann necklace. We define $\pi(\I)$ to be the closed polygonal curve whose vertices are, in order, $\pi(I_1)$, $\pi(I_2)$, \dots, $\pi(I_n)$, $\pi(I_1)$.

\begin{proposition}
The curve $\pi(\I)$ is a simple closed curve.
\end{proposition}

\begin{remark}
When $\M$ is $\binom{[n]}{k}$, the interior of $\pi(\I)$ is convex, but this is not true in general.
\end{remark}

\begin{proof}
Suppose for the sake of contradiction that the curve $\pi(\I)$ crosses through itself. 
First, by assumption, $I_i \neq I_j$ so, by Lemma~\ref{lem:vecinj}, the vertices $\pi(I_i)$ and $\pi(I_j)$ are disjoint.
We now rule out the case that line segments $(\pi(I_i), \pi(I_{i+1}))$ and $(\pi(I_j), \pi(I_{j+1}))$ cross in their interior. 
The case that $(\pi(I_i), \pi(I_{i+1}))$ passes through the vertex $\pi(I_j)$ is similar and easier.

The collection $\I$ is weakly separated by Lemma~\ref{lem:IinM}.
By Proposition~\ref{prop:embed}, the map $\pi$ embeds $\Sigma(\I)$ into $\RR^2$. 
The only way that the line segments  $(\pi(I_i), \pi(I_{i+1}))$ and $(\pi(I_j), \pi(I_{j+1}))$ could cross is if $I_i$, $I_j$, $I_{i+1}$ and $I_j$ are vertices of some two-dimensional face of $\Sigma(\I)$, arranged in that circular order.
(Or the reverse circular order, but then we could switch the indices $i$ and $j$.)
We consider the case that this face is of the form $\W(K)$, the case of $\B(L)$ is similar.
Let $(I_i, I_j, I_{i+1}, I_{j+1}) = (Ka, Kb, Kc, Kd)$; the elements $a$, $b$, $c$ and $d$ must be circularly ordered.

Now, since $i$ is the unique element of $I_i \setminus I_{i+1}$, we know that $a=i$. 
Similarly, $b=j$.
But then $Kb \leq_{i+1} Kc$, and the inequality is strict because $I_{i+1} \neq I_j$. This contradicts that $I_{i+1}$, which is $Kc$, is the $\leq_{i+1}$ minimal element of $\M$.
\end{proof}

The aim of the next several propositions is to establish:
\begin{proposition} \label{Winding}
Let $\I$ be a connected Grassmann necklace. Let $J$ be weakly separated from $\I$, but not an element of $\I$. 
Then $\pi(J)$ is in the interior of $\pi(\I)$ if and only if $J$ is in $\M$.
\end{proposition}

Let $\phi(t)$ be the parameterization of the simple closed curve $\pi(\I)$, with $\phi(r/n) = \pi(I_r)$, and $\phi(r/n + u) = (1-u)  \pi(I_r) + u \pi(I_{r+1})$ for $u \in (0,1/n)$.  
We consider the path $\phi(t) - \pi(J)$ in $\RR^2 \setminus (0,0)$. 
The point $\pi(J)$ is inside $\pi(\I)$ if and only if $\phi(t) - \pi(J)$ has winding number $1$; if not, it has winding number $0$.
We will prove proposition~\ref{Winding} by computing this winding number. 
We first introduce some notation.

By hypothesis, $I_r$ and $J$ are weakly separated. Let $[a_r, b_r]$ be the smallest cyclic interval containing all the elements of $I_r \setminus J$, and $[c_r, d_r]$ the smallest cyclic interval containing all the elements of $J \setminus I_r$. Since $J \neq I_r$, these intervals are nonempty.
So $a_r$, $b_r$, $c_r$ and $d_r$ are circularly ordered, with $b_r \neq c_r$ and $d_r \neq a_r$.
Now, 
$$\pi(I_r) - \pi(J) = \sum_{i \in I_r \setminus J} v_i -  \sum_{j \in J \setminus I_r} v_j.$$
Let $p_r$ be the centroid of $\{ v_i \}_{i \in I_r \setminus J}$, and $q_r$ the centroid of $\{ v_j \}_{j \in J \setminus I_r}$; since $I_r$ and $J$ are weakly separated, $p_r \neq q_r$.
So $\pi(I_r) - \pi(J)$ is parallel to $p_r - q_r$.

\begin{lemma}\label{lem:Intervals}
There are circular intervals $[a'_r, b'_r]$, and $[c'_r, d'_r]$, with $(a'_r, b'_r, c'_r, d'_r)$ circularly ordered, so that $[a_r, b_r]$ and $[a_{r+1}, b_{r+1}]$ are contained in $[a'_r, b'_r]$ and $[c_r, d_r]$ and $[c_{r+1}, d_{r+1}]$ are contained in $[c'_r, d'_r]$.
\end{lemma}

\begin{proof}
We have to consider four cases, depending on whether or not $r$ and $\pi(r)$ are in $J$. 
Let's consider the case that neither of them is. In this case, $[c_r, d_r] = [c_{r+1}, d_{r+1}]$ and we can take $[c'_r, d'_r] = [c_r, d_r]$. 
The intervals $[a_r, b_r]$ and $[a_{r+1}, b_{r+1}]$ are disjoint from $[c_r, d_r]$, and thus live in $(d_r, c_r)$. 
It is thus easy to see that there is a minimal subinterval of $(d_r, c_r)$ which contains $[a_r, b_r]$ and $[a_{r+1}, b_{r+1}]$; we take this interval to be $[a'_r, b'_r]$. 

The other three cases are similar to this one.
\end{proof}

For $t$ a real number between $r/n$ and $(r+1)/n$, define $p(t)$ by linear interpolation between $p_r$ and $p_{r+1}$. 
Similarly, define $q(t)$ by linear interpolation between $q_r$ and $q_{r+1}$. 
For $t$ between $r/n$ and $(r+1)/n$, the point $p(t)$ is in the convex hull of $\{ v_i \}_{i \in [a'_r, b'_r]}$ and $q(t)$ is in the convex hull of $\{ v_j \}_{j \in [c'_r, d'_r]}$. 
Because the intervals $[a'_r, b'_r]$ and $[c'_r, d'_r]$ are disjoint, we know that $p(t) \neq q(t)$.
The vector $\phi(t) - \pi(J)$ is a positive scalar multiple of $p(t) - q(t)$. 

Let $\gamma$ be a convex simple closed curve through the points $v_i$.
Let $x(t)$ and $y(t)$ be the points where the line through $p(t)$ and $q(t)$ meets $\gamma$, with $x(t)$ closer to $p(t)$.
So $x(t)$ lies on the part of $\gamma$ between $v_{a'_r}$ and $v_{b'_r}$; the point $y(t)$ lies on the part of $\gamma$ between $v_{c'_r}$ and $v_{d'_r}$.
Note that $x(t)$ and $y(t)$ are continuous functions of $t$.

It is geometrically clear that the paths $x(t)$ and $y(t)$, travelling around $\gamma$, have the same winding number as the path $x(t) - y(t)$ around $0$.
As observed above, the vector $x(t) - y(t)$ is a positive multiple of $\phi(t) - \pi(J)$. 
Thus, we are reduced to computing the winding number of $x(t)$ around $\gamma$.

\begin{proof}[Proof of Proposition~\ref{Winding}]
First, suppose that $J \in \M$.
This means that, for every $r$, we have $J \geq_r I_r$; since $J \not \in \I$, this inequality is in fact strict.
The indices $(r, a_r, b_r, c_r, d_r)$ are circularly ordered; the inequality between $b_r$ and $c_r$ are strict, while the others are weak.
So the interval $[a'_r, b'_r]$ can only contain $r$ at its left extreme.
For $t$ between $r/n$ and $(r+1)/n$, the point $x(t)$ lies in the interval $[a'_r, b'_r]$, so $x(t)$ has winding number $1$, as desired.

Now, suppose that $J \not \in M$.
Then, for some $s \in [n]$, we have $J \not \geq_{s} I_s$; by hypothesis, $J \parallel I_s$. 
%All inequalities in the rest of this proof refer to $\leq_s$.
So one of the following two cases must hold: 
\begin{enumerate}
\item the minimal element of $J \setminus I_s$ is less than the  minimal element of $I_s \setminus J$ 
or 
\item the maximal element of $I_s \setminus J$ is greater than the maximal element of $J \setminus I_s$.
\end{enumerate}
Here and in the rest of this proof, ``minimal", ``maximal", ``less than" and ``greater than" refer to the orders $\leq_s$.

We discuss the former case; the latter is similar.
We know that $I_s$ is minimal among the sets $I_r$.
So, for every $r$, the minimal element of  $J \setminus I_r$ must be less than the  minimal element of $I_r \setminus J$, in the $\leq_s$ order.

So, for every $r$, the interval $[a_r, b_r]$ does not contain $s$.
So $[a'_r, b'_r]$ also does not contain $s$, and we see that $x(t)$ is not $v_s$, for any $t$. 
So $x(t)$ has winding number $0$, as desired.
\end{proof}

We now apply our constructions to the case that $G$ is a plabic graph.

\begin{theorem} \label{thrm:dual}
Let $G$ be a bipartite plabic graph whose strand permutation is connected and let $\C=\f(G)$. Then $\Sigma(\C)$ is isomorphic to the dual complex to the planar graph $G$.
Also, $\pi(\Sigma(\f(G)))$ fills the region surrounded by $\pi(\I)$. 
\end{theorem}

\begin{remark}
If $G$ is a reduced plabic graph which may not be bipartite, let $G'$ be the graph obtained by contracting edges between vertices of the same color, as discussed in Remark~\ref{rem:Contract}. 
Then the faces of $G$ and $G'$ are in bijection, and they have the same labels. 
Letting $\C$ be this common set of labels, it then follows from Theorem~\ref{thrm:dual} that $\Sigma(\C)$ is dual to $G'$. 
\end{remark}

\begin{remark}
If $G$ is a bipartite plabic graph whose strand permutation is \emph{not connected}, then $\Sigma(\C)$ is a union of $\Sigma(\C')$'s for various smaller collections $\C'$, by Proposition~\ref{GlueAtPoint}. 
Each of these collections comes from a connected plabic graph $G'$, to which Theorem~\ref{thrm:dual} applies. 
\end{remark}

\begin{proof}[Proof of Theorem~\ref{thrm:dual}]
By Proposition~\ref{prop:plagws}, $\C$ is weakly separated, so it makes sense to define $\Sigma(\C)$. 
By definition, the vertices of $\Sigma(\C)$ correspond to the faces of $G$, with a separate vertex for each face of $G$ by Proposition~\ref{prop:PartConverse}.
By Propositions~\ref{prop:embed} and~\ref{Winding}, we know that $\Sigma(\C)$ embeds into $\RR^2$, with the boundary vertices corresponding to $\I$.
Let $R$ be the region in $\RR^2$ surrounded by the boundary of $\pi(\Sigma(\C))$.

Let $D(G)$ be the dual to the planar graph $G$. So $D(G)$ is a two-dimension CW-complex, homemorphic to a disc, with vertices labeled by $\C$, and with boundary labeled by $\I$. We now describe a map $\psi$ of $D(G)$ into $\RR^2$:
For a vertex $v$ of $D(G)$, if $J$ is the label of the dual face of $G$, then $\psi(v)$ is $\pi(J)$. 
For each edge $e$ of $D(G)$, connecting $u$ and $v$, let $\psi(e)$ be a line segment from $\psi(u)$ to $\psi(v)$.

We now discuss faces of $D(G)$.
Let $F$ be a face of $D(G)$, dual to a vertex $v$ of $G$. 
Suppose that $v$ is black; the case where $v$ is white is similar. 
Let the edges ending at $v$ cross the strands $i_1$, $i_2$, \dots, $i_r$.
So the faces bordering $v$ have labels of the form $L \setminus i_1$, $L \setminus i_2$, \dots, $L \setminus i_r$ for some $L$.
In particular, we see that the corresponding vertices of $D(G)$ are in $\B(L)$. 
The face $\B(L)$ is embedded in $\RR^2$ as a convex polygon.
Moreover, the vertices of $F$ occur in cyclic order among those of $\B(L)$.
So $\psi$ embeds $\partial F$ as a convex polygon in $\RR^2$, and we extend $\psi$ to take the interior of $F$ to the interior of that polygon.
Note that we have shown that, for every face $F$ of $D(G)$, the image $\psi(F)$ lies within a face of $\Sigma(\C)$ of the same color. 

The map $\psi$ takes the boundary of $D(G)$ to $\partial R$.
For every face $F$ of $D(G)$, the face $\psi(F)$ is embedded in $\RR^2$ with the correct orientation. 
Since $D(G)$ is homeomorphic to a disc, these points force $\psi$ to be a homeomorphism onto $R$.

Every face of $\psi(D(G))$ lies in a face of $\pi(\Sigma(\C))$, so $\pi(\Sigma(\C)) \supseteq \psi(D(G)) = R$. By Proposition~\ref{Winding}, we also have $\pi(\Sigma(\C)) \subseteq R$, so we conclude that $\pi(\Sigma(\C)) = \psi(D(G)) = R$.
We now must show that the two CW-structures on this subset of $\RR^2$ coincide.

Let $F$ be any face of $D(G)$. 
We consider the case that $F$ is black; the white case is similar.
As shown above, there is a black face $\pi(\B(L))$ of $\Sigma(\C)$ which contains $\psi(F)$.
Suppose, for the sake of contradiction, that some point of $\pi(\B(L))$ is not in $\psi(\Sigma(\C))$. 
Since $D(G)$ covers all of $R$, there must be some face $F'$ of $D(G)$, adjacent to $F$, such that $\psi(F')$ overlaps $\pi(\B(L))$.
But $F'$, being a neighbor of $F$, is white and thus lies in some $\pi(\W(K))$. So $\pi(\B(L))$ and $\pi(\W(K))$ overlap, a contradiction. 
We conclude that every face of $\psi(D(G))$ is also a face of $\pi(\Sigma(\C))$. 
Since $\psi$ is injective, and has image $R$, we see that every face of $\pi(\Sigma(\C))$ is precisely one face of $\psi(D(G))$. 

We now have a bijection between the faces of $\Sigma(\C)$ and of $D(G)$. 
In both $\pi(\Sigma(\C))$ and $\psi(D(G))$, a face is the convex hull of its vertices, so the corresponding faces of $\Sigma(\C)$ and $D(G)$ must have the same vertices and the same edges. 
So $\pi^{-1} \circ \psi$ is a cell-by-cell homeomorphism between the simplicial complexes $\Sigma(\C)$ and $D(G)$.
\end{proof}

\begin{remark} 
The most technical part of this paper is establishing that, when $\C$ is a maximal weakly separated collection which is not yet known to be of the form $\f(G)$, then $\pi(\Sigma(\f(G)))$ fills the interior of the region surrounded by $\pi(\I)$. This task is accomplished in Section~\ref{sec:final}.
\end{remark}

We now show that $\f(G)$ is always maximal.

\begin{theorem} \label{thrm:IsMaximal}
Let $G$ be a reduced plabic graph, with boundary $\I$. Then $\f(\G)$ is a maximal weakly separated collection in $\M_{\I}$.
\end{theorem}

\begin{proof}
We first reduce to the case that $\I$ is connected. If not, let $I_{i} = I_j$. We resume the notations of section~\ref{Connected}. The face of $G$ labeled $I_i$ touches the boundary of the disc at two points; both between boundary vertices $i$ and $i+1$ and between boundary vertices $j$ and $j+1$. So this face disconnects the graph $G$ into two graphs, with vertices in $I^1$ and the other with vertices in $I^2$. Call these plabic graphs $G^1$ and $G^2$. It is an easy exercise that each $G^j$ is a reduced plabic graph, with boundary $\I^j$. By induction on $n$, we know that $\f(G^j)$ is a maximal weakly separated collection in $\M_{\I^j}$. 

Let $J$ in $\M_{\I}$ be weakly separated from $\f(G)$.
Since $\f(G) \cup \{ J \}$ is weakly separated, by Proposition~\ref{SplittingSummary}, $J$ must either be of the form $J^1 \cup I^2$ or $I^1 \cup J^2$, where $J^r$ is weakly separated from $\f(G^r)$ and in $\M_{\I^r}$. 
But, by the maximality of $\f(G^r)$, this implies that $J^r$ is in $\f(G^r)$. Then $J$ is the label of the corresponding face of $\f(G)$.
We have shown that any $J$ in $\M_{\I}$ which is weakly separated from $\f(G)$ must lie in $\f(G)$, so $\f(G)$ is maximal in $\M_{\I}$. 

We now assume that $\I$ is connected.
Let $J$ in $\M_{\I}$ be weakly separated from $\f(G)$ and assume, for the sake of contradiction, that $J \not \in \f(G)$.
We showed above that $\pi(\Sigma(\f(G)))$ fills the interior of the region surrounded by $\pi(\I)$. 
By Proposition~\ref{Winding}, $\pi(J)$ is in this region, so $\pi(J)$ lands on some point of $\pi(\Sigma(\f(G)))$.
In particular, $\pi$ is not injective on $\Sigma(\f(G) \cup \{ J \})$, contradicting the assumption that $J$ is weakly separated from $\C$.
This contradiction shows that $\f(G)$ is maximal in $\M_{\I}$.
\end{proof}

\section{Proof of Lemma~\ref{lem:key}} \label{sec:key}

Our aim in this section is to prove the following technical lemma. 
The reader may wish to skip this proof on first reading, and go on to see how it is used in the proof of Theorem~\ref{thm:WSisPG}.

\begin{lemma}\label{lem:key}
Let $\C$ be a weakly separated collection of $\binom{[n]}{k}$. Let  $E \subset \binom{[n]}{k-1}$ and $x \not = y \in [n]$ be such that $E \cup \{x\}$ and $E \cup \{y\}$ are in $\C$. If $E$ is not the interval $(x,y)$, then there is a set of one of the following two types which is weakly separated from $\C$:
\begin{enumerate}
\item $M^a := E \cup \{a\}$ with $a \in (x,y) \setminus E$ or
\item $N^b := E \cup \{x,y\} \setminus \{b\}$ with $b \in (y,x) \cap E$.
\end{enumerate}
\end{lemma}

Note that this lemma is taken to itself under the symmetry which replaces every subset of $[n]$ by its complement and switches $x$ and $y$. We call this symmetry \newword{dualization}.

%
%Our proof is by contradiction. 
%So, from now on, we are operating under:
%\begin{Assumption} \label{Falsehood}
%There is a maximal weakly separated collection $\C$ that does not contain any $M^a$ or $N^b$. 
%\end{Assumption}

The condition that $E \neq (x,y)$ means that either there exists some $a \in (x,y) \setminus E$ or there exists some $b \in (y,x) \cap E$ (or both). 
%We break symmetry and make the assumption:
%\begin{Assumption} \label{BreakSym}
%There is some $a$ in $(x,y) \setminus E$.
%\end{Assumption}

Consider under what circumstances a set of the form $M^a$ might not be weakly separated from $\C$.
In order for this to happen, there must be some $J$ in $\C$ such that $J \not \parallel M^a$.
The next lemma studies the properties of such a $J$.

\begin{lemma} \label{WitnessDescription}
Let $E$, $x$ and $y$ be as above, and let $a \in (x,y) \setminus E$. 
Suppose that $J \parallel Ex$, $J \parallel Ey$ but $J \not \parallel Ea$. 
%Then $a \not \in J$. There are $\ell$ and $r \in J \setminus E$ such that $x \leq_x \ell <_x a <_x r \leq_x y$, we have $(\ell,r) \cap (J \setminus E) = \emptyset$ and $(r, \ell) \cap (E \setminus J)  \neq \emptyset$.
Then there are elements $\ell$ and $r$ such that:
\begin{enumerate}
\item $\ell$ and $r$ are in $J \setminus E$
\item $x \leq_x \ell <_x r \leq_x y$
\item $(r, \ell) \cap (E \setminus J) \neq \emptyset$.
\item $(\ell, r) \cap (J \setminus E) = \emptyset$.
\item $a \in (\ell, r)$
\end{enumerate}
\end{lemma}

\begin{proof}
Note that $J$ is a vertex of $\Sigma(\C)$, and $(Ex, Ey)$ is contained in a face of  $\Sigma(\C)$. (It may or may not be an edge.) 
Then, by Lemma~\ref{WriteThis}, the vectors $e_J$ and $e_E + (1/2) \left( e_x + e_y \right)$ are weakly separated.
By hypothesis, $e_J$ and $e_E + e_a$ are not weakly separated.
The only coordinates to change between $e_J - \left( e_E + (1/2) \left( e_x + e_y \right) \right)$ and $e_J -e_E - e_a$ are in positions $x$, $y$ and $a$. 
Coordinates $x$ and $y$ increase by $1/2$ from half integers to integers. In particular, they cannot change from nonnegative to negative or nonpositive to positive.
So, the failure of weak separation between $J$ and $Ea$ must be attributable to coordinate $a$.
We must have $a \not \in J$, and we must have $a \in (\ell, r)$ for some $\ell$ and $r \in J \setminus E$ such that $(r, \ell) \cap (E \setminus J) \neq \emptyset$.
Moreover, we choose to take the interval $(\ell, r)$ to be of minimal length, subject to the conditions that $a \in (\ell, r)$ and $l$ and $r \in J \setminus E$. 
This ensures that  $(\ell,r) \cap (J \setminus E) = \emptyset$.

We are left to check that $\ell$ and $r$ are in $[x,y]$.
We check that $\ell$ is; the case of $r$ is similar.
Suppose, instead, that $\ell \in (y,x)$.
If $x \in J$ then $(x,r)$ would be closer to $a$ then $(\ell,r)$, contradicting our minimal choice of $(\ell, r)$.
So $x \not \in J$.
Looking at $(\ell, x, a, y)$, we see that $J$ and $Ey$ are not weakly separated, a contradiction.
\end{proof}

We define $(J, \ell, r)$ to be a \newword{witness against $M^a$} if $J \in \C$ and the numbered conditions of Lemma~\ref{WitnessDescription} hold.

\begin{lemma} \label{OverlappingWitnesses}
Suppose that $(x,y) \setminus E$ is nonempty, and no $M^a$ is weakly separated from $\C$.
Then there is a sequence of triples, $(J_1, \ell_1, r_1)$, \dots, $(J_q, \ell_q, r_q)$, each of which obeys,
\begin{enumerate}
\item $J_m \in \C$
\item $\ell_m$ and $r_m$ are in $J_m \setminus E$
\item $x \leq_x \ell_m <_x r_m \leq_x y$
\item $(r_m, \ell_m) \cap (E \setminus J_m) \neq \emptyset$.
\item $(\ell_m, r_m) \cap (J_m \setminus E) = \emptyset$.
\end{enumerate}
and such that  $x=\ell_1 <_x \ell_2 <_x \cdots <_x \ell_q$ and $r_1 <_x r_2 <_x \cdots <_x r_q=y$ and $\ell_{i+1} <_x r_i$.
\end{lemma}

\begin{proof}
Consider any $a \in (x,y) \setminus E$.
Since $M^a$ is not weakly separated from $\C$, there is some $J$ in $M^a$ which is not weakly separated from $M^a$.
Since $Ex$ and $Ey$ are in $\C$, we know that $J$ is weakly separated from $Ex$ and $Ey$.
So we can complete $J$ to a witness, $(J, \ell, r)$, against $M^a$. This $(J, \ell, r)$ will obey the numbered conditions above.

Choose a minimal collection $(J_1, \ell_1, r_q)$, \dots, $(J_q, \ell_q, r_q)$, obeying the numbered conditions, so that every $a \in (x,y) \setminus E$ lies in some $(\ell_m, r_m)$.
(Since there is some $a \in (x,y) \setminus E$, this minimal collection is nonempty.)
We cannot have $(\ell_i, r_i) \subseteq (\ell_j, r_j)$ for any $i \neq j$, as otherwise we could delete $(J_i, \ell_i, r_i)$ and have a smaller collection.
Thus, we may reorder our collection so that $\ell_1 <_x \ell_2 <_x \ldots <_x \ell_q$ and $r_1 <_x r_2 <_x \cdots <_x r_q=y$.

If $\ell_1 \in (x,y)$, then $\ell_1$ is an element of $(x,y) \setminus E$ and $\ell_1$ does not lie in any $(\ell_m, r_m)$, a contradiction.
Similarly, $r_q \not \in (x,y)$. So $\ell_1=x$ and $r_q = y$.
Finally, suppose that $\ell_{i+1} \geq_x r_i$. 
Then $r_i$ is an element of $(x,y) \setminus E$ which is not in any $(\ell_m, r_m)$, a contradiction.
\end{proof}

We now assume that $(x,y) \setminus E$ is nonempty and that no $M^a$ is weakly separated from $\C$.
For clarity, we will continue to state these assumptions explicitly in all lemmas that rely on them.
We fix, once and for all, a sequence of triples $(J_m, \ell_m, r_m)$ as in Lemma~\ref{OverlappingWitnesses}.

We introduce a total order $\lw$ on $\binom{[n]}{k}$:
Given $I$ and $J$ in $\binom{[n]}{k}$, first compare $| J \cap (x,y) |$ and $ | I \cap (x,y) |$. 
If $| I \cap (x,y) | < | J \cap (x,y) |$, set $I \lw J$.
If this comparison is inconclusive and $|I \cap (y,x) | > | J \cap (y,x)|$, set $I \lw J$.
If both of these comparisons are inconclusive, and $I \cap (x,y)$ is lexicographically before $J \cap (x,y)$, set $I \lw J$; 
if that too is inconclusive, and $I \cap (y,x)$ is lexicographically after $J\cap (y,x)$, set $I \lw J$. (Here we use the orderings $<_x$ and $<_y$ on $(x,y)$ and $(y,x)$ respectively.)
Finally, if all of these are inconclusive\footnote{Which decision we make in this case is completely immaterial to the proof; we just made a decision so that the order would be total.}, then either $I=J$ or $\{ I,J \} = \{ Sx, Sy \}$ for some $S$; set $Sx \lw Sy$.

This order is constructed to obey the dual lemmas:

\begin{lemma} \label{lw-works}
Suppose that $I$ and $J$ are weakly separated. Suppose that there are $a <_a b <_a c$ in $[x,y]$ such that $a$ and $c \in I \setminus J$ and  $b \in J \setminus I$. Then $I \lw J$.
\end{lemma}

\begin{lemma} \label{lw-worksdual}
Suppose that $I$ and $J$ are weakly separated. Suppose that there are $a <_a b <_a c$ in $[y,x]$ such that $a$ and $c \in J \setminus I$ and  $b \in I \setminus J$. Then $I \lw J$.
\end{lemma}

We prove Lemma~\ref{lw-works} and leave it to the reader to dualize our argument to prove Lemma~\ref{lw-worksdual}.

\begin{proof}[Proof of Lemma~\ref{lw-works}]
Since $I$ and $J$ are weakly separated, there can be no element of $[y,x]$ which is in $J \setminus I$.
If there is any element of $I \setminus J$ in $[y,x]$, then $| I \cap (x,y) | < |J \cap (x,y)|$, so we have $I \lw J$ as desired.

If there is no element of $I \setminus J$ in $[y,x]$, then $|I \cap (x,y)| = |J \cap (x,y)|$ and $|I \cap (y,x)| = |J \cap (y,x)|$, so we move to the third and fourth prongs of the test.
We have $I \cap (y,x) = J \cap (y,x)$. We know that $a \in (I \cap (x,y)) \setminus (J \cap (x,y))$, and, because $I$ and $J$ are weakly separated, there can be no element of $(J \cap (x,y)) \setminus (I \cap (x,y))$ which is $\leq_x$-prior to $a$. So $I \lw J$ in this case as well.
\end{proof}

%
%Let $\epsilon$ be a transcendental positive real number, less than $1/k$.
%Define a function $w : [n] \to \RR$ by $w(x) = 1$, $w(x+1) = 1+\epsilon$, $w(x+2) = 1+ \epsilon + \epsilon^2$, \ldots, $w(y-1) = 1+ \epsilon + \cdots + \epsilon^{y-x-1}$, $w(y) = -1$, $w(y+1) = -1 - \epsilon$, \ldots, $w(x-1) = -1-  \epsilon - \ldots - \epsilon^{y-x-1}$. (We must lift the exponents $x-y-1$ and $y-x-1$ to lie between $0$ and $n-1$; they naturally live in $\mathbb{Z}/n$.)
%Define $I \lw J$ if $\sum_{i \in I} w(i) < \sum_{j \in J} w(j)$. Choosing $\epsilon$ transcendental ensures that we never have equality.
%
%\begin{lemma} \label{lw-works}
%Suppose that $I$ and $J$ are weakly separated. Suppose that there are $a <_a b <_a c$ in $[x,y]$ such that $a$ and $c \in I \setminus J$ and  $b \in J \setminus I$. Then $I \lw J$.
%\end{lemma}
%
%
%\begin{proof}
%
%The conditions on $I$ and $J$ force $J \setminus I$ to be contained in $[x,y]$. Let $x+s$ be the $<_x$ least member of $J \setminus I$. Also, there must be some $x+r \in (I \setminus J) \cap [x,y]$ with $x+s <_x x+r$. Then
%\begin{multline*}
%\sum_{j \in J} w(j) - \sum_{i \in I} w(i) \geq \left( \epsilon^{r+1} + \cdots + \epsilon^s \right) - \sum_{j>s} (k-1) \left(\epsilon^{j+1} + \epsilon^{j+2} + \cdots \right) \\ 
% > \epsilon^{r+1} - (k-1) \epsilon^{s+1}/(1-\epsilon) \geq \frac{\epsilon^{r+1}}{1-\epsilon} (1-\epsilon^{s-r}(k-1) - \epsilon).
%\end{multline*}
%The condition that $\epsilon<1/k$ ensures that this is positive.
%\end{proof}

We make the following definition:
For $1 \leq i \leq m \leq j \leq q$, we say that $J_m$ is an \newword{$(i,j)$-snake} if
\begin{enumerate}
 \item All of the integers $\ell_i$, $\ell_{i+1}$, \dots, $\ell_m$, $r_m$, \dots, $r_{j-1}$, $r_j$ are in $J_m$ and
\item $(r_j, \ell_i) \cap (E \setminus J_m)$ is nonempty.
\end{enumerate}

Note that $J_m$ is always an $(m,m)$-snake. The following lemma allows us to construct snakes.

\begin{lemma} \label{SnakeExists}
Assume that $(x,y) \setminus E$ is nonempty and that no $M^a$ is weakly separated from $\C$.

For any indices $i$ and $j$ with $1 \leq i \leq j \leq q$, let $J_m$ be the $\lw$-minimal witness among $J_i$, $J_{i+1}$, \dots, $J_j$. Then $J_m$ is an $(i,j)$-snake.
In particular, for any $(i,j)$, there is an $(i,j)$-snake.
\end{lemma}

\begin{proof}
Our proof is by induction on $j-i$. When $i=j$, the result is clear.

We now consider the case that $i<j$. We first establish that  $\ell_i$, $\ell_{i+1}$, \dots, $\ell_m$, $r_m$, \dots, $r_{j-1}$, and $r_j$ are in $J_m$.
Suppose, for the sake of contradiction, that $r_m$, $r_{m+1}$, \dots, $r_{j'-1}$ are in $J_m$ but $r_{j'}$ is not.
By induction, there is a $(m+1, j')$ snake $J_{m'}$.
Then $\ell_{m+1}$ and $r_{j'}$ are in $J_{m'}$ and not in $J_m$, while $r_{m'-1}$ is in $J_{m}$ and not $J_{m'}$ by Lemma~\ref{lw-works}.
This means that $J_{m'} \lw J_m$, a contradiction.
We conclude that  $r_m$, \dots, $r_{j-1}$, and $r_j$ are in $J_m$; similarly,  $\ell_i$, \dots, $\ell_{m-1}$  and $\ell_m$ are in $J_m$.

We now establish that $(r_j, \ell_i) \cap (E \setminus J_m)$ is nonempty.
We know that $(r_m, \ell_m) \cap (E \setminus J_m)$ is nonempty; let $e \in (r_m, \ell_m) \cap (E \setminus J_m)$. 
Suppose, for the sake of contradiction, that $(r_j, \ell_i) \cap (E \setminus J_m)$ is empty, and hence $e \in (r_{j'-1}, r_{j'})$ for some $j'$.
(The case where $e \in (\ell_{i'},  \ell_{i'+1})$ is similar.)
By induction, there is an $(m+1, j')$-snake $J_{m'}$.
Let $e'$ be an element of $(r_{j'}, \ell_{m+1}) \cap (E \setminus J_{m'})$.

$Ex$ and $J_{m'}$ are weakly separated.
Considering $(\ell_{m+1}, e, r_{j'}, e')$, we see that $e \in J_{m'}$.
But then $\ell_{m+1}$ and $e$ are in $J_{m'}$ and not $J_m$, while $r_{m'-1}$ and $e'$ are in $J_m$ and not $J_{m'}$. By our choice of $j'$ such that $e \in (r_{j'-1}, r_{j'})$, we know that $r_{m'} \subset (\ell_{m+1}, e). $So $J_{m'}$ and $J_m$ are not weakly separated, a contradiction.

%by our choice of $j'$ such that $e \in (r_{j'-1}, r_{j'})$, we know that $r_{m'} \subset (\ell_{m+1}, e)$.
%So $J_{m'} \lw J_m$, a contradiction.
\end{proof}

As a corollary, we remove the need to assume that $(x,y) \setminus E$ is nonempty.

\begin{proposition}
Assume that $\C$ is not weakly separated from any $M^a$ nor any $N^b$, and recall the standing hypothesis that $E$ is not the interval $(x,y)$. Then both $(x,y) \setminus E$ and $(y,x) \cap E$ are nonempty.
\end{proposition}

\begin{proof}
By our standing hypotheses, either $(x,y) \setminus E$ or $(y,x) \cap E$ is nonempty.
Without loss of generality, assume it is the former so the $(J_m, \ell_m, r_m)$ are defined.
Lemma~\ref{SnakeExists} applies, and we know that there is an $(1,q)$-snake $J_m$.
For this $J_m$, we know that $(r_q, \ell_1) \cap (E \setminus J_m)$ is nonempty.
Since $(r_q, \ell_1) \cap (E \setminus J_m) = (y,x) \cap (E \setminus J_m) \subseteq (y,x) \cap E$, we conclude that $(y,x) \cap E$ is also nonempty.
\end{proof}

We are now entering into the final stages of the proof of Lemma~\ref{lem:key}.
Assume, for the sake of contradiction, that $\C$ is not weakly separated from any $M^a$ nor any $N^b$.
So we have a sequence of witnesses $(J_m, \ell_m, r_m)$, and a dual sequence $(J'_m, \ell'_m, r'_m)$.
We point out that $\ell'_m$ and $r'_m$ are in $[y,x]$; we leave it to the reader to dualize the other parts of the definition of a witness.

Let $J_{\mu}$ and $J'_{\mu'}$ be the $\lw$-minimal elements among the sequences $J_i$ and $J'_i$.

\begin{lemma} \label{MusAreWide}
 $J_{\mu}$ contains both $x$ and $y$, and $(y,x) \cap (E \setminus J_{\mu})$ is nonempty. 
$J'_{\mu'}$ contains neither $y$ nor $x$ and $(x,y) \cap (J'_{\mu'} \setminus E)$ is nonempty.
\end{lemma}

\begin{proof}
The first statement is Lemma~\ref{SnakeExists} applied to $(i,j) = (1,q)$. 
The second statement is the dual of the first.
\end{proof}

Let $H := (x,y) \cap (J'_{\mu'} \setminus E)$. 
By Lemma~\ref{MusAreWide}, $H$ is nonempty; let $a \in H$.
By definition, $a \in (x,y) \setminus E$, so there is some $(J_{i}, \ell_{i}, r_{i})$ which is a witness against $M^a$. 
In particular, $H \cap (\ell_i, r_i)$ contains $a$, and is thus nonempty.

Let $(J_{\nu}, \ell_{\nu}, r_{\nu})$ be $\lw$-minimal such that $H \cap (\ell_{\nu}, r_{\nu})$ is nonempty.

\begin{proposition}
With the above definition, $J_{\nu} \lw J'_{\mu'}$.
\end{proposition}

\begin{proof}
For the sake of contradiction, assume that $J_{\nu} \not \lw J'_{\mu'}$.

Let $h$ be in $H \cap (\ell_{\nu}, r_{\nu})$.
Then $h$ is in $J'_{\mu'}$ and not in $J_{\nu}$.
By the contrapositive of lemma~\ref{lw-works}, either $J'_{\mu'}$ contains every element of $J_{\nu} \cap [x,h)$, or $J'_{\mu'}$ contains every element of $J_{\nu} \cap (h,y]$.
Without loss of generality, assume that it is the latter.

Since $y \not \in J'_{\mu'}$ (Lemma~\ref{MusAreWide}), we deduce that $y \not \in J_{\nu}$. Now, $y=r_q$.
Let $j$ be such that $r_{\nu}$, $r_{\nu+1}$, \dots, $r_{j-1}$ are in $J_{\nu}$ and $r_j$ is not.

Let $J_m$ be the $\lw$-minimal element of $J_{\nu}$, $J_{\nu + 1}$, \dots, $J_{j}$.
By Lemma~\ref{SnakeExists}, $J_m$ is a $(\nu, j)$-snake.
Since $r_j \not \in J_{\nu}$, we know that $J_m \neq J_{\nu}$ and thus we have the strict inequality $J_m \lw J_{\nu}$.

Since $\nu \leq m-1 <j$, we have $r_{m-1} \in J_{\nu}$. 
Since $J'_{\mu'} \supseteq J_{\nu} \cap (h,y]$, we have $r_{m-1} \in J'_{\mu'}$.
Also, one of the conditions on the witness $(J_{m-1}, \ell_{m-1}, r_{m-1})$ is that $r_{m-1} \not \in E$.
So $r_{m-1}$ is in $H$.

But $r_{m-1}$ is in $(\ell_m, r_m)$ and we saw above that $J_m \lw J_{\nu}$.
This contradicts the minimality of $J_{\nu}$.
\end{proof}

We have now shown that there is some $J_{\nu}$ with $J_{\nu} \lw J'_{\mu'}$.
By the minimality of $J_{\mu}$, we have $J_{\mu} \lweq J_{\nu}$, so $J_{\mu} \lw J'_{\mu'}$.
But the dual argument shows that $J'_{\mu'} \lw J_{\mu}$.
We have reached a contradiction, and Lemma~\ref{lem:key} is proved. \qedsymbol

\section{Proof of the Purity Conjecture} \label{sec:final}

In this section, we prove Theorem~\ref{thm:GSConjecture}. 
More precisely, we prove:

\begin{theorem} \label{thm:WSisFGPos}
Let $\I$ be a Grassmann necklace, $\M$ the associated positroid.
Let $\C$ be a subset of $\binom{[n]}{k}$.
Then $\C$ is a maximal weakly separated collection in $\M$ if and only if $\C$ is of the form $\f(G)$ for a plabic graph $G$ whose decorated permutation, $\pi_G$ corresponds to the Grassmann necklace $\I$.
\end{theorem}

See section~\ref{sec:Grneck} to review the definitions of positroids, Grassmann necklaces and decorated permutations.
%As discussed in section REFERENCE, this implies Theorem~\ref{thm:GSConjecture} and thus the conjectures CITE, CITE, CITE and CITE.

By Theorem~\ref{thrm:IsMaximal}, if $\C$ is of the form $\f(G)$, then $\C$ is a maximal weakly separated collection in $\M$.
So our goal is to prove the converse.
Suppose that $\C$ is a maximal weakly separated collection in $\M$.

We first reduce to the case that $\I$ is connected.
Suppose that $I_i=I_j$. 
We reuse the notations $I^1$, $I^2$ and so forth from section~\ref{Connected}.
So, by Proposition~\ref{SplittingSummary}, $\C^1$ and $\C^2$ are maximal weakly separated collections in  $\M^1$ and $\M^2$. 
By induction, there are reduced plabic graphs $G^1$ and $G^2$, with $f(G^r) = \C^r$,  and with boundary regions labeled $I^1$ and $I^2$. 
Gluing these graphs along these boundary regions, we have a reduced plabic graph $G$ with $f(G) = \C$.

Thus, we may now assume that $\I$ is connected.
Form the $CW$-complex $\Sigma(\C)$. 
By Proposition~\ref{prop:embed}, the map $\pi$ embeds $\Sigma(\C)$ into $\RR^2$. 
By Proposition~\ref{Winding}, for every $J \in \C \setminus \I$, the point $\pi(J)$ is inside the curve $\pi(\I)$.
Since these are the vertices of $\Sigma(\C)$, we see that $\Sigma(\C)$ lies within this curve.

Thus, $\Sigma(\C)$ is a finite polyhedral complex in $\RR^2$, with outer boundary $\pi(\I)$.
We will eventually take $G$ to be the dual graph to $\Sigma(\C)$.
In order to do this, we need the following proposition:

\begin{proposition} \label{NoHoles}
If $\C$ is a maximal weakly separated collection for the connected positroid $\M$, then $\Sigma(\C)$ fills the entire interior of the curve $\pi(\I)$.
\end{proposition}

\begin{figure}
\centerline{\scalebox{0.7}{
\psfrag{123}[cc][cc]{\Huge $123$}
\psfrag{234}[cc][cc]{\Huge $234$}
\psfrag{345}[cc][cc]{\Huge $345$}
\psfrag{456}[cc][cc]{\Huge $456$}
\psfrag{156}[cc][cc]{\Huge $156$}
\psfrag{126}[cc][cc]{\Huge $126$}
\psfrag{125}[cc][cc]{\Huge $125$}
\psfrag{135}[cc][cc]{\Huge $135$}
\psfrag{235}[cc][cc]{\Huge $235$}
\psfrag{U}[cc][cc]{\Huge $U$}
\psfrag{Ex}[cc][cc]{\Huge $Ex$}
\psfrag{Ey}[cc][cc]{\Huge $Ey$}
\includegraphics{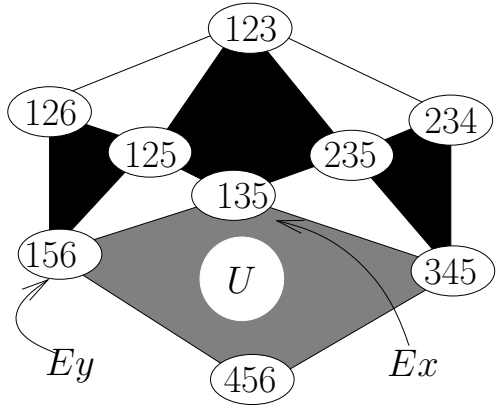}}}
\caption{A plabic tiling for a non-maximal $\C$} \label{Incomplete}
\end{figure}

\begin{example} \label{IncompleteExample}
Consider the weakly separated collection 
$$\C=(123, 234, 345, 456, 156, 126, 125, 135, 235)$$
 for the Grassmann necklace $(123, 234, 345, 456, 156, 126)$. 
This collection only has $9$ elements, not $10$ which a maximal collection would have.
Figure~\ref{Incomplete} shows the corresponding $\Sigma(\C)$, embedded in $\RR^2$.
The grey region, labeled $U$, is not in $\Sigma(\C)$, but note that all four edges of its boundary are in $\Sigma(\C)$.
In the vocabulary of the proof of Proposition~\ref{NoHoles}, we take $E=\{ 1, 5 \}$, $x=3$ and $y=6$. Our strategy is to use Lemma~\ref{lem:key} to show that $\C$ must be weakly separated from one of $145$, which is $E \cup 4 $, or $356$, which is $Exy \setminus 1$. 
(In fact, it is weakly separated from both.)
\end{example}

\begin{proof}[Proof of Proposition~\ref{NoHoles}]
Let $N$ denote the interior of $\pi(\I)$.
Suppose, to the contrary, that $U$ is a connected component of $N \setminus \pi(\Sigma(\C))$. 
Let $(\pi(Ex), \pi(Ey))$ be an edge in $\partial U$, with $U$ on the lefthand side as we look from $\pi(Ex)$ to $\pi(Ey)$. 
Note that $(Ex, Ey)$ is not of the form $(I_{i}, I_{i+1})$ as, if it were, then $U$ would be on the outside of $\pi(\I)$. In particular, $E$ is not the interval  $(y,x)$.

By Lemma~\ref{lem:key}, there is a set $J$, either of the form $Ea$, with $a \in (x,y) \setminus E$, or the form $Exy \setminus b$, with $b \in (y,x) \cap E$, such that $J$ is weakly separated from $\C$. The conditions on $a$ and $b$ imply that the triangle $(\pi(Ex), \pi(Ey), \pi(J))$ lies on the $U$ side of $(\pi(Ex), \pi(Ey))$. (The reader is invited to check this in Example~\ref{IncompleteExample}.)
If $J \in \M$ then $J$ must be in $\C$, by the maximality of $\C$.
But then the triangle $(\pi(Ex), \pi(Ey), \pi(J))$ is part of $\Sigma(\C)$, and lies on the $U$-side of $(\pi(Ex), \pi(Ey))$, contradicting the supposition that $(\pi(Ex), \pi(Ey))$ is a boundary edge of $U$.
 So $J \not \in \M$. (At this point, we are already done if $\M$ is the uniform matroid.)

We have shown that $J \not \in \M$.
By proposition~\ref{Winding}, $\pi(J)$ lies outside of $\pi(\I)$.
In particular, the triangle $(\pi(Ex), \pi(Ey), \pi(J))$ must cross some boundary edge $(\pi(I_i), \pi(I_{i+1}))$.
Let's say that segments $(\pi(Ex), \pi(J))$ and $(\pi(I_i), \pi(I_{i+1}))$ cross (the case of $Ey$ is practically identical).
Note that $(Ex, J)$ and $(I_i, I_{i+1})$ are both subsets of cliques in the weakly separated collection $\C \cup \{ J \}$. 
It almost contradicts Proposition~\ref{prop:embed} for these two segments to cross.
The only subtlety is that $(Ex, J)$ and $(I_i, I_{i+1})$ might be contained in the same clique of $\C \cup \{ J \}$, with $(Ex, I_i, J, I_{i+1})$ occurring in that circular order around the corresponding face of $\Sigma(\C \cup \{ J \})$. 
But then the triangle $(Ex, I_i, I_{i+1})$ is contained in $\Sigma(\C)$, and lies on the $U$-side of $(Ex, Ey)$, contradicting that $(Ex, Ey)$ is supposed to be on the boundary of $U$.
\end{proof}

So, $\Sigma(\C)$ is a disc, whose two faces are naturally colored black and white. 
Let $G$ be the dual graph, so the faces of $G$ are labeled by $\C$, and the boundary faces by $\I$.
There are $n$ vertices on the boundary of $G$, labeled by $[n]$, with vertex $i$ between faces $I_i$ and $I_{i+1}$.

We now show that $G$ is reduced. Be warned that, at this point, we do not know that the labels of the faces of $G$ inherited from $\Sigma(\C)$ are the labels $\f(G)$.
Let $\gamma$ be a strand which separates $Ea$ and $Ex$, for some $x$, and then goes on to separate $Ea$ from some $Ey$.
As we travel along $\gamma$, the faces on the  left are of the form $E_i a$ for some sequence $E_i$, with $|E_i \setminus  E_{i+1}| = 1$, and the faces on the right are of the form $E_{i} \cup E_{i+1}$. 
%We leave it to the reader to check that $E_i <_a E_{i+1}$. 

We have $E_i <_a E_{i+1}$ due to the definition of cliques in plabic tilings. Therefore, $\gamma$ can not loop or cross itself; it must travel from the boundary of $N$ to the boundary of $N$. 
In particular, we see that there are only $n$ strands in total, one ending between each pair of boundary faces of $G$.

The strand $\gamma$ must start separating $I_{\pi^{-1}(a)}$ from $I_{\pi^{-1}(a)+1}$, and end separating $I_a$ from $I_{a+1}$.
That is because these are the only pairs of adjacent faces at the boundary of $G$ which differ by an $a$.
Thus, the strand we have called $\gamma$ is the strand labeled by $a$ in our standard way of labeling the strands of a plabic graph.

Now fix two strands $\alpha$ and $\beta$ which intersect at least twice, labeled $a$ and $b$. We want to show that the intersections occur in reverse order along $\alpha$ and $\beta$. At one crossing, let the strands separate faces $Sa$ and $Sb$; at the other crossing let the adjacent faces be $Ta$ and $Tb$. From the fact that $Sa \parallel Tb$ and $Sb \parallel Ta$, we get that one of the following holds:
\begin{itemize}
\item $S \setminus T \subset (a,b)$ and $T \setminus S \subset (b,a)$ or
\item $T \setminus S \subset (a,b)$ and $S \setminus T \subset (b,a)$.
\end{itemize}

Assuming the first case holds, we get $S <_a T$ and $T <_b S$. So the intersections occur in opposite order along $\alpha$ and $\beta$. 
We also conclude that the intersections occur in opposite order in the second case.
We have now checked that strands do not cross themselves, do not form closed loops, and that the intersections along any pair of strands occur in opposite order along the two strands. 
So $G$ is reduced.

Finally, we must check that the labels coming from $\Sigma(\C)$ are the same as the labels $\f(G)$. 
In this paragraph, when we refer to the label of a face of $G$, we mean the label $I$ of the dual vertex of $\Sigma(\C)$.
Consider a strand $\gamma$, with label $a$. 
We must check that the faces to the left of $\gamma$ all contain $a$, and those to the right do not.
The faces immediately to the left of $\gamma$ all contain $a$, and those immediately to the right do not.
Let $\Delta$ be the union of all faces that contain $a$, so $\gamma$ is part of the boundary of $\Delta$, as is part of $\partial N$.
We claim that, in fact, this is the entire boundary of $\Delta$. If there were any other edge $e$ in the boundary of $\Delta$, since $G$ is a disc with boundary $\partial N$, there would have to be a face on the other side of $e$ from $\Delta$, and the label of this face would not contain $a$. But then one of the strands passing through this edge would be $\gamma$. Similarly, the boundary of $G \setminus \Delta$ is also made up of $\gamma$ and a piece of $\partial N$. It is now topologically clear that $\Delta$ is the part of $N$ on one side of $\gamma$, as desired.

So starting from a maximal weakly separated collection $\C$ for $\M$, we have obtained a reduced plabic graph $G$ with $\f(G) = \C$. This completes the proof of Theorem~\ref{thm:WSisFGPos}. \qedsymbol.

\smallskip

We have now proved Theorem~\ref{thm:WSisFGPos}. 
By Theorem~\ref{thm:P}, we now know that every maximal weakly separated collection of $\M_{\I}$ has cardinality $\ell(\I) + 1$ and any two maximal weakly separated collections of $\M$ are mutation equivalent.
This is Theorem~\ref{thm:GSConjecture} and our work is complete.

Before we end, we state a simple result that follows from Theorem~\ref{thm:WSisFGPos}.

\begin{corollary}
\label{cor:WSAlign}
Let $\I$ be a Grassmann necklace, $\pi^:$ the associated decorated permutation, and $\M$ the associated positroid. We define $\PCh(\I)$ to be the union of $\F(G)$ for all reduced plabic graphs $G$ of $\M$. Then $J \in \M$ is in $\PCh(\I)$ if and only if it obeys the following condition: For any alignment $\{i,j\}$ in $\pi$, if $i \in J$ then we have $j \in J$.
\end{corollary}

\begin{proof}
One direction is obvious: For any $J \in \F(G)$ for some plabic graph $G$ of $\M$, it should meet the condition due to the way the faces are labeled in plabic graphs. 

So we only need to check the other direction. Pick any $J \in \M$ that satisfy the condition. Assume for the sake of contradiction that $J$ is not weakly separated with some $I_i$. This means that there is some $k \in J \setminus I_i$ and $t \in I_i \setminus J$ such that $i,k,t$ is circularly ordered. The condition $k \not \in I_i$ tells us that $\pi^{-1}(k) \leq_i k$, while the condition $t \in I_i$ tells us that $\pi^{-1}(t) \geq_i t$. So $i,\pi^{-1}(k),k,t,\pi^{-1}(t)$ are circularly ordered and $\{k,t\}$ forms an alignment of $\pi$. Then $J$ should contain both $k$ and $t$ due to the condition, but this contradicts $t \not \in J$. So $J$ is weakly separated from all $I_i$'s, and $\{I_1,\cdots,I_n,J\}$ is a weakly separated collection. Then Theorem~\ref{thm:WSisFGPos} tells us that there is some $\F(G)$ that contains $J$.

\end{proof}

\section{Connection with $w$-chamber sets of Leclerc and Zelevinsky}
\label{sec:Connection_with_LZ}

In this section, we state Leclerc and Zelevinsky's conjectures on $w$-chamber sets and explain how they follow from Theorem~\ref{thm:GSConjecture}.

Given two sets of integers, $I$ and $J$, we write $I \LZprec J$ if $i<j$ for all $i \in I$ and all $j \in J$. We will write $\LZprec_i$ for $\LZprec$ with respect to the shifted order $<_i$ on $[n]$.
The following definition will only be necessary in this section; it is an extension of the definition of weak separation to sets of unequal cardinality.

\begin{definition}
We say that $I,J \subset [m]$ are weakly separated in the sense of Leclerc and Zelevinsky if at least one of the following holds:
\begin{enumerate}
\item $|I| \geq |J|$ and $J - I$ can be partitioned into a disjoint union $J \setminus I = J' \sqcup J''$ so that $J' \LZprec I \setminus J \LZprec J''.$
\item $|J| \geq |I|$ and $I - J$ can be partitioned into a disjoint union $I \setminus J = I' \sqcup I''$ so that $I' \LZprec J \setminus I \LZprec I''.$
\end{enumerate}

We will write $I \LZpar J$ to denote that $I,J$ are weakly separated (in the sense of Leclerc and Zelevinsky). 
\end{definition}

For example, $\{1,2\} \LZpar \{3\}$ but $\{1,4\} \not \LZpar \{3\}$.

\begin{definition}
Let $w$ be a permutation in the symmetric group $S_m$.
The $w$-chamber set, $\Ch(w)$, is defined to be the collection of those subsets $I$ of $[m]$ 
that satisfy the following condition:  
for every pair $a < b$ with $w(a) < w(b)$, if $a \in I$ then $b \in I$.
\end{definition}

\begin{remark}
The attentive reader will notice that our definition of a $w$-chamber set differs from that
in~\cite{LZ} by switching the directions of all inequalities.  Thus, the
$w$-chamber set in our notation is the image of Leclerc and Zelevinsky's $(w_0
w w_0)$-chamber set under the map $i \mapsto n+1-i$.  This flip has to be
included somewhere, due to an incompatibility in the sign conventions
of~\cite{LZ} and~\cite{Post}.
\end{remark}

\begin{definition} 
A weakly separated collection $\C$ of $\Ch(w)$ is a
collection of subsets in $\Ch(w)$ such that the elements of $\C$ are pairwise
weakly separated (in the LZ sense). A maximal weakly separated collection of
$\Ch(w)$ is a weakly separated collection of $\Ch(w)$ that is maximal under
inclusion.  
\end{definition}

Here is the original conjecture of Leclerc and Zelevinsky.

\begin{theorem} 
\label{thm:LZConjecture}
Any maximal weakly separated collection of $\Ch(w)$ is of cardinality
$n+\ell(w) + 1$, where $\ell(w)$ is the length of $w$.  
\end{theorem}

Subsets of $[m]$ of any cardinality can be identified with certain $m$ element
subsets of $[2m]$ through a simple padding construction, as follows.

\begin{definition}
Given $I \subset [m]$, we define $\pad{I}$ to be $I \cup \{2m,\cdots,m+|I| -1 \}$. Given a collection $\C \subset 2^{[m]}$, we define $\pad{\C}$ to be the collection of $\pad{I}$ for each $I \in \C$.
\end{definition}

For example, if $m=4$, then $\pad{\{1,3\}} = \{1,3,7,8 \}$. 

Let $n = 2m$ and $k=m$.  The padding allows
us to reduce weak separation in the sense of Leclerc and Zelevinsky to the
notion of weak separation of elements in $\binom{[n]}{k}$ which we use in the rest of the paper.
Specifically, we have:
\begin{lemma}
Given $I$ and $J \subset [m]$, we have $I \LZpar J$ if and only if $\pad{I} \parallel \pad{J}$.
\end{lemma}

\begin{proof}
If $|I|=|J|$ then $I\LZpar J$ if and only if $I\parallel J$ if and only if $\pad{I} \parallel \pad{J}$.
Otherwise, we may assume that $|I| > |J|$. 
%the other two cases are very similar.
Then $\pad{J} \setminus \pad{I} = (J \setminus I) \cup \{ m+|J|-1, m+|J|, \ldots, m+|I|-1 \}$. Note that $J \setminus I$ is contained in $[m]$, as is $\pad{I} \setminus \pad{J}$. So $\pad{I}$ is weakly separated from $\pad{J}$ if and only if there is no element of $J \setminus I$ which lies between two elements of $I \setminus J$ in the standard linear order on $[m]$.  This precisely means that we can partition $J \setminus I$ as $J' \sqcup J''$ so that $J' \LZprec I \setminus J \LZprec J''$.
\end{proof}

Next, we explain the relation between positroids and $w$-chamber sets.  For a permutation $w\in S_m$, let
$\I(\hat{w})=(I_1,\dots,I_{2m})$ be the Grassmann necklace corresponding to the decorated permutation
$$
\hat{w} \ := \  [2m,2m-1,\cdots,m+1,w^{-1}(m),\cdots,w^{-1}(1)].
$$ 
(The permutation $\hat{w}$ has no
fixed points, so we do not need to describe a coloring.)

\begin{lemma}
\label{lem:samelength}
For any $w \in S_m$, we have $\ell(w) + m+1 = \ell(\I(\hat{w}))+1$. Furthermore, for $J \subseteq [m]$, we have $J \in \Ch(w)$ if and only if $\pad{J} \in \PCh(\I(\hat{w}))$.
\end{lemma}

\begin{proof}

Let us compute the number of alignments of $\hat{w}$. All $\binom{m}{2}$ pairs $\{i,j\}$ with $i,j\in[m]$ form an alignment. A pair $\{2m+1-i,2m+1-j\}$ with $i,j\in[m]$ forms an alignment if and only if $(i,j)$ is not an inversion of $w^{-1}$, that is, $i<j$ and $w^{-1}(i)<w^{-1}(j)$.
There are $\binom{m}{2}-\ell(w^{-1}) = \binom{m}{2} - \ell(w)$ alignments of this type. Finally, a pair $\{i,m+j\}$ with $i,j\in[m]$ never forms an alignment. In total, we get
$\ell(\I(\hat{w})) = m^2 - A(\hat{w}) = m^2 - 2 \binom{m}{2} + \ell(w) = m + \ell(w).$

Now let's check the second arugment. Any pair $a<b$ with $w(a) < w(b)$ translates into an alignment $\{a,b\}$ of $\hat{w}$. So $\pad{J} \in \PCh(\I(\hat{w}))$ implies $J \in \Ch(w)$. The alignments $\{a,b\}$ of $\hat{w}$ such that $a,b \in [m+1,2m]$ translates into the restriction that any $H \in \PCh(\I(\hat{w}))$ is obtained by padding some subset of $[m]$. And as we have checked above, a pair $\{i,m+j\}$ with $i,j\in[m]$ never forms an alignment. So $J \in \Ch(w)$ implies $\pad{J} \in \PCh(\I(\hat{w}))$.
\end{proof}

Combining the above two lemmas, we have obtained the following claim.

\begin{proposition} \label{samews} $\C$ is a maximal weakly separated
collection of a $w$-chamber set $\Ch(w)$ if and only if
$\pad{\C}$ is a maximal
weakly separated collection inside the positroid $\M_{\I(w)}$.  \end{proposition}

%We leave it for the reader to check that $\ell(w) + m+1 = \ell(\I(\hat{w}))+1$.

%\begin{lemma} \label{samelength}
%For any $w \in S_m$, we have $\ell(w) + m+1 = \ell(\I(\hat{w}))+1$.
%\end{lemma}

%\begin{proof}
%Let us compute the number of alignments of $\hat{w}$.
%All ${m\choose 2}$ pairs $\{i,j\}$ with $i,j\in[m]$ form an alignment.
%A pair $\{2m+1-i,2m+1-j\}$ with $i,j\in[m]$ forms an alignment
%if and only if $(i,j)$ is not an inversion of $w^{-1}$,
%that is, $i<j$ and $w^{-1}(i)<w^{-1}(j)$.
%There are ${m\choose 2}-\ell(w^{-1}) = {m\choose 2} - \ell(w)$ alignments of this type.  
%Finally, a pair $\{i,m+j\}$ with $i,j\in[m]$ never forms an alignment.
%In total, we get
%$\ell(\I(\hat{w})) = m^2 - A(\hat{w}) = m^2 - 2 \binom{m}{2} + \ell(w) = m + \ell(w).$
%\end{proof}

%Combining Proposition~\ref{samews} and Lemma~\ref{lem:samelength}, we see that Theorem~\ref{thm:LZConjecture} follows 
%immediately from Theorem~\ref{thm:GSConjecture}.

So Theorem~\ref{thm:LZConjecture} follows 
immediately from Theorem~\ref{thm:GSConjecture}.

%\bibliographystyle{plain}
%\bibliography{purity}

\begin{thebibliography}{WWWW}

\bibitem{QA}
Arkady Berenstein and Andrei Zelevinsky.
\newblock Quantum Cluster Algebras.
\newblock \emph{Advances in Mathematics} \textbf{195} no. 2, 405 -- 455,  (2005). 


\bibitem{Bock}
Raf Bocklandt.
\newblock Calabi Yau algebras and weighted quiver polyhedra.
\newblock \texttt{arXiv:0905.0232}

\bibitem{Broom}
Nathan Broomhead.
\newblock Dimer models and Calabi-Yau algebras.
\newblock \texttt{arXiv:0901.4662}

\bibitem{Bry}
Thomas Brylawski.
\newblock Constructions.
\newblock Chapter  7 of \emph{Theory of Matroids}, Encyclopedia of Mathematics and its Applications Volume 26, edited by Neil White, 1986.

\bibitem{DKK}
Vladimir~I. Danilov, Alexander~V. Karzanov, and Gleb~A. Koshevoy.
\newblock On maximal weakly separated set-systems, 2009.
\newblock \emph{Journal of Algebraic Combinatorics}, \textbf{32},  497--531,  2010.

\bibitem{GLS}
Christof Geiss, Bernard Leclerc and Jan Schr\"oer.
\newblock Cluster structures on quantum coordinate rings, 2011.
\newblock \texttt{arXiv:1104.0531}

\bibitem{KLS}
Allen Knutson, Thomas Lam, and David~E. Speyer.
\newblock Positroid varieties: juggling and geometry, 2009.
\newblock \texttt{arXiv:0903.3694}

\bibitem{LZ}
Bernard Leclerc and Andrei Zelevinsky.
\newblock Quasicommuting families of quantum Pl\"ucker coordinates.
\newblock In {\em Advances in Math. Sciences (Kirillov's seminar), AMS
  Translations}, \textbf{181} 85--108, 1998.

\bibitem{Post}
Alexander Postnikov.
\newblock Total positivity, grassmannians, and networks, 2006.
\newblock \texttt{http://math.mit.edu/$\sim$apost/papers/tpgrass.pdf}

\bibitem{Scott}
Joshua Scott.
\newblock Quasi-commuting families of quantum minors.
\newblock {\em Journal of Algebra}, \textbf{290(1)}, 204 -- 220, 2005.

\bibitem{Scott2}
Joshua Scott.
\newblock Grassmannians and cluster algebras.
\newblock {\em Proceedings of the London Mathematical Society}, \textbf{92(3)}, 345--380, (2006).

\bibitem{Simion}
Rodica Simion.
\newblock Noncrossing partitions
\newblock \emph{Discrete Mathematics} \textbf{217} no. 1--3, 367--409, 2000.

\end{thebibliography}

\end{document}